\documentclass[12pt]{amsart}
\usepackage{stmaryrd}
\usepackage{}
\usepackage{amsmath}
\usepackage{amsfonts}
\usepackage{amssymb}
\usepackage[all,cmtip]{xy}           
\usepackage{xspace}
\usepackage{bbding}
\usepackage{txfonts}
\usepackage[shortlabels]{enumitem}
\usepackage{ifpdf}
\ifpdf
  \usepackage[colorlinks,final,backref=page,hyperindex]{hyperref}
\else
  \usepackage[colorlinks,final,backref=page,hyperindex,hypertex]{hyperref}
\fi
\usepackage{tikz}
\usepackage[active]{srcltx}

\makeatletter

\newtheorem{theorem}{Theorem}[section]
\newtheorem{prop}[theorem]{Proposition}
\newtheorem{lemma}[theorem]{Lemma}
\newtheorem{coro}[theorem]{Corollary}
\newtheorem{prop-def}{Proposition-Definition}[section]

\theoremstyle{definition}
\newtheorem{defn}[theorem]{Definition}

\newtheorem{remark}[theorem]{Remark}
\newtheorem{exam}[theorem]{Example}

\newcommand{\nc}{\newcommand}

\nc{\delete}[1]{{}}
\nc{\mmargin}[1]{}

\nc{\mlabel}[1]{\label{#1}}  
\nc{\mcite}[1]{\cite{#1}}  
\nc{\mref}[1]{\ref{#1}}  
\nc{\meqref}[1]{\eqref{#1}} 
\nc{\mbibitem}[1]{\bibitem{#1}} 

\delete{
\nc{\mlabel}[1]{\label{#1}  
{\hfill \hspace{1cm}{\bf{{\ }\hfill(#1)}}}}
\nc{\mcite}[1]{\cite{#1}{{\bf{{\ }(#1)}}}}  
\nc{\mref}[1]{\ref{#1}{{\bf{{\ }(#1)}}}}  
\nc{\meqref}[1]{\eqref{#1}{{\bf{{\ }(#1)}}}} 
\nc{\mbibitem}[1]{\bibitem[\bf #1]{#1}} 
}

\nc{\twrba}{that is a twisted Rota-Baxter algebra\xspace}
\nc{\bfn}{\mathbf{n}}

\nc{\shadow}{phantom\xspace}
\nc{\Shadow}{Phantom\xspace}
\nc{\shad}{\theta}
\nc{\tforall}{\text{ for all }}
\nc{\ddiff}{D-\text{differential}\xspace}
\nc{\oplin}{operator linear\xspace}
\nc{\fopav}{F_A(\Omega,V)}
\nc{\tw}{twisted\xspace}
\nc{\wtd}{weighted\xspace}
\nc{\wt}{weight\xspace}
\nc{\wte}{\lambda}
\nc{\wtddiff}{shifted differential\xspace}
\nc{\coa}{\overline{\mathfrak{OA}}} 
\nc{\oa}{\mathfrak{OA}} 

\nc{\mrba}{MRBA\xspace}
\nc{\mrbas}{MRBAs\xspace}
\nc{\mtrba}{MTRBA\xspace}		
\nc{\mtrbas}{MTRBAs\xspace}		
\nc{\smtrba}{MTRBA\xspace}
\nc{\smtrbas}{MTRBAs\xspace}

\nc{\fopaa}{F(\Omega,C)}
\nc{\fopab}{F_A(\Omega,B)}
\nc{\algw}{\frakS}	
\nc{\opw}{\Delta}	
\nc{\ralgw}{\frakI}	
\nc{\ropw}{\Pi}		
\nc{\sett}{\calt}	
\nc{\algt}{\frakT}	
\nc{\opt}{\Lambda}	
\nc{\rsett}{\cale}	
\nc{\ralgt}{\frakE}	
\nc{\ropt}{\Gamma}	

\nc{\tprod}{{\scriptsize{\veebar}}}
\nc{\graftprod}{grafting product\xspace}
\nc{\extnop}{extension operator\xspace}
\nc{\frakT}{\mathfrak{T}}

\nc{\frakJ}{\mathfrak{J}}
\nc{\frakI}{\mathfrak{I}}

\nc{\eva}{\mathrm{eva}}

\nc{\mrho}{\rho}
\nc{\mfraka}{\tau}
\nc{\rtwist}{\tau}
\nc{\stmap}{j}

\nc{\frmtrba}{F_{\mathrm{\mtrba}}}
\nc{\frmrba}{F_{\mathrm{MRBA}}}
\nc{\mj}{j}
\nc{\freet}[1]{\free{#1}_{\mathrm{MTRBA}}}
\nc{\freem}[1]{\free{#1}_{\mathrm{MRBA}}}

\nc{\trsha}{\frmrba}

\nc{\reypr}{\circledast} 	

\nc{\lc}{\lfloor} \nc{\rc}{\rfloor}
\nc{\free}[1]{\overline{#1}}
\nc{\Id}{\mathrm{Id}}
\nc{\lra}{\longrightarrow}
\nc{\hra}{\hookrightarrow}

\nc{\rsha}{\sha^{\rm rel}}

\newcommand{\bk}{{\mathbf{k}}}

\nc{\vep}{\varepsilon}
\nc{\bin}[2]{ (_{\stackrel{\scs{#1}}{\scs{#2}}})}  
\nc{\binc}[2]{(\!\! \begin{array}{c} \scs{#1}\\
    \scs{#2} \end{array}\!\!)}  
\nc{\bincc}[2]{  ( {\scs{#1} \atop
    \vspace{-1cm}\scs{#2}} )}  
\nc{\bs}{\bar{S}}
\nc{\ra}{\longleftarrow}
\nc{\ot}{\otimes}
\nc{\rar}{\rightarrow}
\nc{\dar}{\downarrow}
\nc{\dap}[1]{\downarrow \rlap{$\scriptstyle{#1}$}}
\nc{\defeq}{\stackrel{\rm def}{=}}
\nc{\dis}[1]{\displaystyle{#1}}
\nc{\dotcup}{\ \displaystyle{\bigcup^\bullet}\ }
\nc{\hcm}{\ \hat{,}\ }
\nc{\hts}{\hat{\otimes}}
\nc{\hcirc}{\hat{\circ}}
\nc{\lleft}{[}
\nc{\lright}{]}
\nc{\curlyl}{\left \{ \begin{array}{c} {} \\ {} \end{array}
    \right .  \!\!\!\!\!\!\!}
\nc{\curlyr}{ \!\!\!\!\!\!\!
    \left . \begin{array}{c} {} \\ {} \end{array}
    \right \} }
\nc{\longmid}{\left | \begin{array}{c} {} \\ {} \end{array}
    \right . \!\!\!\!\!\!\!}
\nc{\ora}[1]{\stackrel{#1}{\rar}}
\nc{\ola}[1]{\stackrel{#1}{\la}}
\nc{\scs}[1]{\scriptstyle{#1}} \nc{\mrm}[1]{{\rm #1}}
\nc{\dirlim}{\displaystyle{\lim_{\longrightarrow}}\,}
\nc{\invlim}{\displaystyle{\lim_{\longleftarrow}}\,}
\nc{\dislim}[1]{\displaystyle{\lim_{#1}}} \nc{\colim}{\mrm{colim}}
\nc{\mvp}{\vspace{0.3cm}} \nc{\tk}{^{(k)}} \nc{\tp}{^\prime}
\nc{\ttp}{^{\prime\prime}} \nc{\svp}{\vspace{2cm}}
\nc{\vp}{\vspace{8cm}}
\nc{\modg}[1]{\!<\!\!{#1}\!\!>}
\nc{\intg}[1]{F_C(#1)}
\nc{\lmodg}{\!<\!\!}
\nc{\rmodg}{\!\!>\!}
\nc{\cpi}{\widehat{\Pi}}
\nc{\sha}{{\mbox{\cyr X}}}  
\nc{\ssha}{{\mbox{\cyrs X}}} 
\nc{\tsha}{{\mbox{\cyrt X}}}
\nc{\shai}{{\stackrel{\ra}{\sha}}}
\nc{\shpr}{\diamond}    
\nc{\labs}{\mid\!}
\nc{\rabs}{\!\mid}

\nc{\dr}{\frakR}
\nc{\cdr}{\frakC\frakR}
\nc{\er}{\frakW}
\nc{\cer}{\frakC\frakW}
\nc{\DR}{\overline{\frakR}}
\nc{\CDR}{\overline{\frakC\frakR}}
\nc{\ER}{\overline{\frakW}}
\nc{\CER}{\overline{\frakC\frakW}}

\font\cyr=wncyr10
\font\cyrs=wncyr7
\font\cyrt=wncyr5

\nc{\ann}{\mrm{ann}}
\nc{\Aut}{\mrm{Aut}}
\nc{\can}{\mrm{can}}
\nc{\Cont}{\mrm{Cont}}
\nc{\rchar}{\mrm{char}}
\nc{\cok}{\mrm{coker}}
\nc{\dtf}{{R-{\rm tf}}}
\nc{\dtor}{{R-{\rm tor}}}

\nc{\Div}{{\mrm Div}}
\nc{\End}{\mrm{End}}
\nc{\Ext}{\mrm{Ext}}
\nc{\Fil}{\mrm{Fil}}
\nc{\Fr}{\mrm{Fr}}
\nc{\Frob}{\mrm{Frob}}
\nc{\Gal}{\mrm{Gal}}
\nc{\GL}{\mrm{GL}}
\nc{\Hom}{\mrm{Hom}}
\nc{\hsr}{\mrm{H}}
\nc{\hpol}{\mrm{HP}}
\nc{\id}{\mrm{id}}
\nc{\im}{\mrm{im}}
\nc{\incl}{\mrm{incl}}
\nc{\length}{\mrm{length}}
\nc{\mforall}{\quad \text{for all }}
\nc{\mchar}{\rm char}
\nc{\mpart}{\mrm{part}}
\nc{\ql}{{\QQ_\ell}}
\nc{\qp}{{\QQ_p}}
\nc{\rank}{\mrm{rank}}
\nc{\rcot}{\mrm{cot}}
\nc{\rdef}{\mrm{def}}
\nc{\rdiv}{{\rm div}}
\nc{\rtf}{{\rm tf}}
\nc{\rtor}{{\rm tor}}
\nc{\res}{\mrm{res}}
\nc{\SL}{\mrm{SL}}
\nc{\Spec}{\mrm{Spec}}
\nc{\tor}{\mrm{tor}}
\nc{\Tr}{\mrm{Tr}}
\nc{\tr}{\mrm{tr}}

\nc{\bfk}{{\bf k}}
\nc{\bfone}{{\bf 1}}
\nc{\bfzero}{{\bf 0}}
\nc{\Diff}{\mathbf{Diff}}
\nc{\FMod}{\mathbf{FMod}}
\nc{\Int}{\mathbf{Int}}
\nc{\Mon}{\mathbf{Mon}}
\nc{\remarks}{\noindent{\bf Remarks: }}
\nc{\Rep}{\mathbf{Rep}}
\nc{\Rings}{\mathbf{Rings}}
\nc{\Sets}{\mathbf{Sets}}

\nc{\BA}{{\mathbb A}}   \nc{\CC}{{\mathbb C}}
\nc{\DD}{{\mathbb D}}   \nc{\EE}{{\mathbb E}}
\nc{\FF}{{\mathbb F}}   \nc{\GG}{{\mathbb G}}
\nc{\HH}{{\mathbb H}}   \nc{\LL}{{\mathbb L}}
\nc{\NN}{{\mathbb N}}   \nc{\PP}{{\mathbb P}}
\nc{\QQ}{{\mathbb Q}}   \nc{\RR}{{\mathbb R}}
\nc{\TT}{{\mathbb T}}   \nc{\VV}{{\mathbb V}}
\nc{\ZZ}{{\mathbb Z}}   \nc{\TP}{\widetilde{P}}

\nc{\cala}{{\mathcal A}}    \nc{\calc}{{\mathcal C}}
\nc{\cald}{\mathcal{D}}     \nc{\cale}{{\mathcal E}}
\nc{\calf}{{\mathcal F}}    \nc{\calg}{{\mathcal G}}
\nc{\calh}{{\mathcal H}}    \nc{\cali}{{\mathcal I}}
\nc{\call}{{\mathcal L}}    \nc{\calm}{{\mathcal M}}
\nc{\caln}{{\mathcal N}}    \nc{\calo}{{\mathcal O}}
\nc{\calp}{{\mathcal P}}    \nc{\calr}{{\mathcal R}}
\nc{\cals}{{\mathcal S}}    \nc{\calt}{{\mathcal T}}
\nc{\calw}{{\mathcal W}}    \nc{\calx}{{\mathcal X}}
\nc{\CA}{\mathcal{A}}

\nc{\fraka}{{\mathfrak a}}
\nc{\frakb}{{\mathfrak b}}
\nc{\frakc}{{\mathfrak c}}
\nc{\frakd}{{\mathfrak d}}
\nc{\frake}{{\mathfrak e}}
\nc{\fraki}{{\mathfrak i}}
\nc{\frakj}{{\mathfrak j}}
\nc{\frakk}{{\mathfrak k}}
\nc{\frakB}{{\frak B}}
\nc{\frakC}{{\frak C}}
\nc{\frakE}{{\frak E}}
\nc{\frakF}{{\frak F}}
\nc{\frakm}{{\frak m}}
\nc{\frakM}{{\frak M}}
\nc{\frakp}{{\frak p}}
\nc{\frakR}{{\frak R}}
\nc{\frakS}{{\frak S}}
\nc{\frakA}{{\frak A}}
\nc{\frakW}{{\frak W}}
\nc{\fraku}{{\frak u}}
\nc{\frakv}{{\frak v}}
\nc{\frakw}{{\frak w}}
\nc{\frakx}{{\frak x}}
\nc{\fraky}{{\frak y}}

\newcommand{\ctone}[1]
{\begin{picture}(10,5)(-2,-1)
\put(0,0){\circle*{3}}
\put(3,-2){\tiny $#1$}
\end{picture}}

\newcommand{\cttwo}[3]{\begin{picture}(20,20)(0,-1)
\thicklines
\put(3,0){\circle*{3}}
\put(3,0){\line(0,1){17}}
\put(3,17){\circle*{3}}
\put(6,-3){\tiny $#1$}
\put(6,18){\tiny $#2$}
\put(6,8){\tiny $#3$}
\end{picture}}

\newcommand{\ctthree}[5]{
\begin{picture}(30,20)(-5,-1)
\thicklines
\put(3,0){\circle*{3}}
\put(3,0){\line(1,2){8}}
\put(11,16){\circle*{3}}
\put(-5.5,16){\circle*{3}}
\put(3,0){\line(-1,2){8}}
\put(6,-4){\tiny $#1$}
\put(-12,16){\tiny $#2$}
\put(14,16){\tiny $#3$}
\put(-8,7){\tiny $#4$}
\put(10,7){\tiny $#5$}
\end{picture}}

\newcommand{\ctfour}[7]{
\begin{picture}(30,20)(-5,-1)
\thicklines
\put(3,0){\circle*{3}}
\put(13,16){\circle*{3}}
\put(3,16){\circle*{3}}
\put(-7,16){\circle*{3}}
\put(3,0){\line(0,1){16}}
\put(3,0){\line(2,3){10}}
\put(3,0){\line(-2,3){10}}
\put(5,-4){\tiny $#1$}
\put(-14,14){\tiny $#2$}
\put(3,20){\tiny $#3$}
\put(17,16){\tiny $#4$}
\put(-10,7){\tiny $#5$}
\put(4,8){\tiny $#6$}
\put(12,7){\tiny $#7$}
\end{picture}}
\newcommand{\ctfourb}[7]{\begin{picture}(40,20)(-5,-1)
\thicklines
\put(3,0){\circle*{3}}
\put(3,0){\line(1,2){6}}
\put(9,12){\circle*{3}}
\put(-3,12){\circle*{3}}
\put(3,0){\line(-1,2){6}}
\put(9,24){\circle*{3}}
\put(9,12){\line(0,1){12}}
\put(2,-7){\tiny $#1$}
\put(-5,2){\tiny $#5$}
\put(-10,8){\tiny $#2$}
\put(8,2){\tiny $#6$}
\put(13,12){\tiny $#3$}
\put(4,16){\tiny $#7$}
\put(13,24){\tiny $#4$}
\end{picture}}

\newcommand{\ctfourc}[7]{
\begin{picture}(30,20)(-5,-1)
\thicklines
\put(3,0){\circle*{3}}
\put(11,24){\circle*{3}}
\put(3,12){\circle*{3}}
\put(-5,24){\circle*{3}}
\put(3,0){\line(0,1){11}}
\put(3,12){\line(2,3){7}}
\put(3,12){\line(-2,3){7}}
\put(5,-4){\tiny $#1$}
\put(1,17){\tiny $#2$}
\put(-12,22){\tiny $#3$}
\put(14,22){\tiny $#4$}
\put(-5,4){\tiny $#5$}
\put(-8,13){\tiny $#6$}
\put(8,13){\tiny $#7$}
\end{picture}}

\newcommand{\ctfive}[9]{\begin{picture}(40,40)(-5,-1)
\thicklines
\put(3,0){\circle*{3}}
\put(3,0){\line(1,2){6}}
\put(9,12){\circle*{3}}
\put(-3,12){\circle*{3}}
\put(3,0){\line(-1,2){6}}
\put(3,24){\circle*{3}}
\put(9,12){\line(1,2){6}}
\put(9,12){\line(-1,2){6}}
\put(15,24){\circle*{3}}
\put(2,-8){\tiny $#1$}
\put(-7,3){\tiny $#6$}
\put(-9,13){\tiny $#2$}
\put(7,1){\tiny $#7$}
\put(12,8){\tiny $#3$}
\put(0,14){\tiny $#8$}
\put(-4,22){\tiny $#4$}
\put(15,14){\tiny $#9$}
\put(20,24){\tiny $#5$}
\end{picture}}

\nc{\ynr}[1]{\textcolor{blue}{\underline{Yunnan:}#1 }}

\nc{\lir}[1]{\textcolor{red}{\underline{Li:}#1 }}

\nc{\rgr}[1]{\textcolor{orange}{\underline{Richard:}#1 }}

\begin{document}

\title[Algebraic study of integral equations]{An algebraic study of Volterra integral equations and their operator linearity}

\author{Li Guo}
\address{Department of Mathematics and Computer Science, Rutgers University, Newark, NJ 07102, United States}
\email{liguo@rutgers.edu}

\author{Richard Gustavson}
\address{Department of Mathematics, Manhattan College, Riverdale, NY 10471, United States}
\email{rgustavson01@manhattan.edu}

\author{Yunnan Li}
\address{School of Mathematics and Information Science, Guangzhou University,
Guangzhou 510006, China}
\email{ynli@gzhu.edu.cn}

\date{\today}

\begin{abstract}
The algebraic study of special integral operators led to the notions of Rota-Baxter operators and shuffle products which have found broad applications. This paper carries out an algebraic study of general integral operators and equations, and shows that there are rich algebraic structures underlying Volterra integral operators and the corresponding equations. First Volterra integral operators are shown to produce a matching \tw Rota-Baxter algebra satisfying twisted integration-by-parts operator identities. In order to provide a universal space to express general integral equations, free operated algebras are then constructed in terms of bracketed words and rooted trees with decorations on the vertices and edges.
Further explicit constructions of the free objects in the category of matching \tw Rota-Baxter algebras are obtained by a twisted and decorated generalization of the shuffle product, providing a universal space for separable Volterra equations. As an application of these algebraic constructions, it is shown that any integral equation with separable Volterra kernels is operator linear in the sense that the equation can be simplified to a linear combination of iterated integrals.
\end{abstract}

\subjclass[2010]{
16W99,  
12H05,	
45N05,  
17B38,	
45D05,  
45P05,   
16S10,  
05C05   
}

\keywords{Integral equation, iterated integral, Volterra operator, Volterra equation, Rota-Baxter algebra, differential algebra, operated algebra, rooted trees, linearity of integral equation}

\maketitle

\tableofcontents

\allowdisplaybreaks

\section{Introduction}

This paper sets up an algebraic framework for integral equations by operated algebras and establishes the operator linearity of separable Volterra integral equations by applying matching \tw Rota-Baxter algebras.
More precisely, this paper provides
\begin{enumerate}
	\item a precise framework to discuss integral equations algebraically in terms of bracketed words and decorated rooted trees. Indeed, the setup applies to equations of any linear operators. This leads to a notion of integral polynomials and hence, by setting the polynomials to zero, integral equations;
	\item an effective framework for integral equations with separable Volterra integral operators, through uncovering algebraic relations of such operators. As an application, we show that any separable integral equation can be expressed in terms of iterated integrals.
\end{enumerate}

\subsection{Background of integral operators and integral equations}
\subsubsection{Analytic aspect}
In general terms, an integral equation is an equation involving integral operators of unknown functions.
The study of integral equations was begun by Fredholm~\mcite{Fr} and Volterra~\mcite{Vo1}, and the two main branches of equations within the discipline, where both limits of integration are fixed and where one or both limits are allowed to vary, bear their respective names. In particular, Volterra equations are defined by the Volterra integral operators
\begin{equation} \mlabel{eq:volop}
P_K(f)(x):=P_{K,a}(f)(x) := \int_a^x K(x,t)f(t)\,dt,
\end{equation}
with kernels $K(x,t)$.

Integral equations and operators have seen many applications since their inception. Hilbert~\mcite{Hi} used integral equations to begin the study of functional analysis.  Integral operators appear when solving boundary value problems through the use of Green's functions, which are the kernels of such operators~\mcite{SH}. See~\mcite{Wa,Ze} for the general theory of integral equations.

Originated from the work of K.-T. Chen, iterated integrals are important in topology and (differential and algebraic) geometry~\mcite{Ch,Ch2,Ha}, and in number theory~\mcite{IKZ} where they are also called Drinfeld integrals or Kontsevich integrals, as well as in quantum field theory~\mcite{Br}. In general, it is a fundamental problem to determine when an integral expression or an integral equation can be expressed in terms of iterated integrals.

\subsubsection{Algebraic aspect}
The study of integral equations from an algebraic perspective was more recent.  One approach has been through the lens of Rota-Baxter algebras, in which a linear operator $P$ must satisfy the Rota-Baxter identity of weight $\lambda$ for some fixed scalar $\lambda$
\begin{equation}
P(f)P(g) = P(fP(g)) + P(P(f)g) + \lambda P(fg).
\mlabel{eq:rbo}
\end{equation}
When $\lambda = 0$, this can be taken as an abstraction of the integration by parts identity for the usual single-variable integral by letting $P(f)(x) = \int_a^x f(t)\,dt$; letting $\lambda$ vary allows Rota-Baxter algebras to be applied in many areas outside of integral equations.  Indeed Rota-Baxter algebras were first used by G. Baxter in fluctuation theory~\mcite{Ba}, and by G.-C. Rota in combinatorics~\mcite{Ro}.  The algebraic structure for the space of integral operators $\int_a^x f(t)\,dt$ using shuffle products was developed from this perspective and is closely related to the aforementioned iterated integrals started by K.-T. Chen~\mcite{Gub}.  From another perspective, integro-differential algebras were developed to algebraically study boundary problems for linear ordinary differential equations~\mcite{GGR, GRR, RR}.

Other algebraic identities are known to hold for special integral operators. One distinguished example is the Reynolds operator identity
\begin{equation}
R(f)R(g) = R(fR(g)) + R(R(f)g) - R(R(f)R(g)) \quad \text{for all } f,g \in A,
\mlabel{eq:rey2}
\end{equation}
satisfied by the Volterra integral operator $P(f)(x):=\int_a^x e^{x-t}\,f(t)\,dt$ which served as one of the first examples of the Reynolds operator originated from the famous work of O. Reynolds on turbulence theory in fluid mechanics~\mcite{Rey} and was interested to R.~Birkhoff and G.-C.~Rota~\mcite{Bi,RoR,Ro1}.

A recent example is the matching Rota-Baxter operators~\mcite{ZGG}, consisting of a family of linear operators $P_\omega, \omega\in \Omega,$ for a parameter set $\Omega$ satisfying the identities
\begin{equation}
\mlabel{eq:mrbo}
P_\alpha(f)P_\beta(g) = P_\alpha(fP_\beta(g)) + P_\beta(P_\alpha(f)g) + \lambda_\beta P_\alpha(fg) \quad \text{for all } f,g \in A,\, \alpha,\beta \in \Omega.
\end{equation}
They are satisfied by a family of integral operators $P_\omega(f)(x):=\int_a^xh_\omega(t)f(t)\,dt$ for a parametric family of functions $h_\omega(t)$.

Despite all these developments, the overall understanding is still very limited on the algebraic aspect of integral operators and integral equations, in comparison with their differential counterparts which, through the pioneering work of Ritt and Kolchin, and contributions of many authors over the recent decades, has evolved into a vast field of theory and applications~\mcite{Ko,PS,Ri}.  Specifically, the algebraic definition of a differential operator is long established as the Leibniz rule or product formula.  Additionally, the meaning of a differential equation from an algebraic perspective is also well established as the ring of differential polynomials on a set of ``differential" variables, namely the polynomials on the given variables and their differential derivations.

\subsection{Integral polynomials and integral equations}
The purpose of this paper is to give precise meanings to the notions of integral equations and integral algebras, serving as the foundation of an algebraic study and symbolic computations of integral equations, analogous to their differential counterparts.

In Section~\mref{sec:volt}, we give an overview of the analytic Volterra operators and equations and show they have the algebraic property that, if the kernel is {\bf separable}, i.e. is a product of single-variable functions, then the algebra of continuous functions equipped with a family $\{P_\omega\,|\,\omega\in \Omega\}$ of separable Volterra operators is a matching Rota-Baxter algebra twisted by invertible elements $\mfraka_\omega$ (Theorem~\mref{thm:twfunction}),
that is, the Volterra integral operators $P_\omega, \omega\in \Omega$, satisfy identities of the form
\[P_\alpha(f)P_\beta(g) = \mfraka_\alpha P_\beta\left(\mfraka_\alpha^{-1}P_\alpha(f)g\right)+\mfraka_\beta P_\alpha\left(\mfraka_\beta^{-1}fP_\beta(g)\right)
\quad \text{for all} f, g\in C(I), \alpha,\beta\in \Omega.\]
Their relationship with the Reynolds operator is briefly discussed.

In Section~\mref{sec:gen} we develop an algebraic structure for defining arbitrary integral equations.  An integral equation should be an element of a suitably defined ``integral polynomial algebra," namely the free commutative ``integral" algebra on a set $Y$ of unknown functions, analogous to the polynomial algebra $\bfk [Y]$ or the differential polynomial algebra $A\{Y\}$.  We aim to find a suitable notion of integral algebra which, in contrast to differential algebras, can have different definitions depending on the nature of the integral operators and might not have any algebraic characterizations. To be versatile, we begin by imposing no conditions on the integral operators except their linearity, in which case, such an algebra is called an operated algebra, defined in~\cite{Gop} but can be traced back to A.G. Kurosh~\mcite{Ku} who called it an $\Omega$-algebra.
The construction of free (noncommutative) operated algebras generated by sets was obtained in~\cite{Gop} in terms of bracketed words, Motzkin paths and rooted trees.

In order to provide a general algebraic framework for operator equations and, in particular, integral equations, we adapt to the commutative case the construction of free operated algebras by bracketed words (Theorem~\mref{thm:freeaab}) and rooted trees (Theorem~\mref{thm:freeobtr}), the latter analogous to the noncommutative case but utilizing typed decorated trees, that is, rooted trees whose vertices and edges are both decorated, arising from the recent study of renormalization~\mcite{BHZ,Foi}.  To handle integral equations with variable coefficients, we consider free commutative operator algebras on sets with coefficients in a preassigned operated algebra, in particular in a ``coefficient algebra" of functions equipped with integral operators. This free object, called the {\bf free relative $\Omega$-operated $\bfk$-algebra} over an algebra $C$, provides the formal and rigorous context for operator equations, and can also be described using bracketed words (Theorem~\mref{thm:freea}) or rooted trees (Corollary~\mref{co:freet}).  Then any integral equation, regarded as a special case of an operator equation, comes from an element of the thus-defined free commutative operated algebra with coefficients in an algebra with integral operators.

When one or a family of linear operators satisfies certain operator identities, an element in the corresponding operator polynomial algebra can descend to a smaller algebra with more concise description, similar to the usual differential polynomials.
Indeed, setting $C = A[Y]$ for some function space $A$ and $\Omega$ a set of integral operators, the free relative $\Omega$-operated algebra is thus the {\bf integral polynomial algebra} with integral operator set $\Omega$, coefficient function space $A$, and variable functions $Y$.  This allows us to not only rigorously define solutions to algebraic integral equations, but also describe when two such equations are equivalent to each other; see Definition~\mref{de:eqn}(\mref{it:eqny}) and Proposition~\mref{pp:reduce}.

In Section~\mref{sec:sep} we apply this abstract approach to the concrete setting of Volterra integral operators as in Eq.~\meqref{eq:volop}, in particular separable Volterra integral operators, for which the kernel $K(x,t)$ is separable.  Applying the \tw Rota-Baxter relation satisfied by the Volterra operators allows us to obtain another construction of the free objects (Theorem~\mref{th:rfcwr}). As an application, we show in Theorem~\mref{thm:oplin} that any integral equation with separable kernels can be equivalently expressed in terms of iterated integrals. Examples are given throughout to illustrate how our algebraic results apply to specific Volterra equations.

\smallskip

\noindent
{\bf Notations.} In this paper, we fix a ground field $\bfk$ of characteristic 0. All the structures
under discussion, including vector spaces, algebras and tensor products, are taken over $\bfk$ unless otherwise specified. Likewise, all algebras are assumed to be commutative unitary $\bfk$-algebras.

\section{Analytic and algebraic operators for integral equations}
\mlabel{sec:volt}
In this section we give a first discussion of the relationship between analytically defined operators from integrals and algebraically defined operators by algebraic operator identities, before their general study in later sections. We show that a Volterra integral operator is a Rota-Baxter operator for only phantom kernels. For separable kernels, it can be realized by a more general operator that we introduce, called \tw Rota-Baxter operators.

\begin{defn}
\mlabel{de:volop}
Let $I\subseteq \RR$ be an open interval and let $C(I)$ be the algebra of continuous functions on $I$.
\begin{enumerate}
\item
A {\bf Volterra operator} is a linear operator
\[
P_K:=P_{K,\,a}:C(I) \to C(I), P_K(f)(x) := \int_a^x K(x,t)f(t)\,dt,
\]
for some $a\in I$ and {\bf kernel} $K(x,t) \in C(I^2)$;
\item
A kernel $K$ and the corresponding Volterra operator are called {\bf separable} if it can be decomposed as $K(x,t) = k(x)h(t)$ for some functions $k$ and $h$ in $C(I)$;
\item
A kernel $K$ and the corresponding Volterra operator are called {\bf \shadow} if it is a function of only the variable of integration (i.e. the ``dummy" variable). It is the special case of a separable kernel when $k(x)$ is a constant;
\item
An integral equation is called {\bf separable} (resp. {\bf \shadow}) if all the integral operators in question are separable (resp. \shadow) and share the same lower limit.
\end{enumerate}
\end{defn}

By imposing suitable convergent conditions on $C(I)$ at the endpoints of $I$, we can also take $a$ to be one of the endpoints of $I$, allowing us to discuss the Thomas-Fermi equation below.

\begin{exam} \mlabel{ex:classic}
\begin{enumerate}
\item \mlabel{it:volt}
Volterra's population model for a species in a closed system~\mcite{Wa2} describes the population $u(t)$ of a species when exposed to both crowding and toxicity; it can be written as an integral equation of the form
\begin{equation}
u(t) = u_0 + a\int_0^t u(x)\,dx - b\int_0^t u(x)^2\,dx - c\int_0^t u(x)\int_0^x u(y)\,dy\,dx,
\mlabel{eq:volt}
\end{equation}
where $a,b$, and $c$ are the birth rate, crowding coefficient, and toxicity coefficient, respectively, and $u_0$ is the initial population.  Here there is one integral operator
\begin{equation} \mlabel{eq:pop}
P:f(t) \mapsto \int_0^t f(x)\,dx.
\end{equation}
\item
\mlabel{it:tf}
The Thomas-Fermi equation~\mcite{Wa1} describes the potential $y(x)$ of an atom in terms of the radius $x$, and can be written as an integral equation of the form
\begin{equation}
y(x) = 1 + Bx + \int_0^x\int_0^t s^{-1/2}y(s)^{3/2}\,ds\,dt,
\mlabel{eq:tf}
\end{equation}
where $B$ is a known parameter.  In this setting, there are two integral operators
\begin{equation}\mlabel{eq:pot}
P_1:f(x) \mapsto \int_0^x f(t)\,dt, \quad P_2:f(x) \mapsto \int_0^x t^{-1/2}f(t)\,dt.
\end{equation}
\end{enumerate}
\end{exam}
Considering the importance of linearity of algebraic and differential equations, we introduce an analogous notion for integral equations in terms of iterated integrals~\mcite{Br,Ch,Ch2}.

\begin{defn} \mlabel{de:oplin}
An integral equation is called {\bf operator linear} if any monomial in the equation is an iterated integral, that is, it does not contain any products of integral operators.
\end{defn}

Note the difference between an equation being {\em operator} linear and linear in the usual sense, namely the unknown function only appears once to the first power in every term of the equation. For example, Eq.~\meqref{eq:volt} is not linear since it contains $u^2$, but it is \oplin since there is no product of integral operators. The two integrals in the last term are iterated, not multiplied. Similarly, a term $yP(y)$ is operator linear, but not linear.
To the contrary, the equation
\begin{equation}
 f(x)+ \Big(\int_0^x (\sin t) (f(t)^3-g(t))\,dt\Big)\Big(\int_0^x \exp(t) g(t)^2\,dt\Big)=0
\mlabel{eq:integ2}
\end{equation}
is not \oplin (and nonlinear) because the product of the Volterra operators.

As an application of our algebraic study, we will show that integral equations with separable kernels can be reduced to one that is \oplin (Theorem~\mref{thm:oplin}).

We next recall the algebraic notions arising from Volterra operators. They will be supplemented by other notions (Definitions~\mref{de:trba} and \mref{de:opalg}) as the algebraic structures to study integral equations in this paper.

\begin{defn}
\begin{enumerate}
\item \mlabel{it:rb}
For any $\lambda \in \bfk$, a {\bf Rota-Baxter algebra of weight} $\lambda$ is a $\bfk$-algebra $R$ together with a linear operator $P:R \to R$ satisfying
\[
P(f)P(g) = P(fP(g)) + P(P(f)g) + \lambda P(fg) \quad \text{for all } f,g \in R.
\]
\item \mcite{GGZy,ZGG} Fix a non-empty set $\Omega$ and let $\lambda_\Omega :=(\lambda_\omega \,|\, \omega \in \Omega) \subseteq \bfk$.  An {\bf $\Omega$-matching Rota-Baxter algebra of weight} $\lambda_\Omega$, or simply an {\bf $\Omega$-MRBA} is a $\bfk$-algebra $R$ together with a family of linear operators $P_\Omega := (P_\omega \,|\, \omega \in \Omega)$, where $P_\omega:R \to R$ for all $\omega \in \Omega$, satisfying
\begin{equation}\mlabel{eq:mrb}
P_\alpha(f)P_\beta(g) = P_\alpha(fP_\beta(g)) + P_\beta(P_\alpha(f)g) + \lambda_\beta P_\alpha(fg) \quad \text{for all } f,g \in R,\, \alpha,\beta \in \Omega.
\end{equation}
If $\lambda_\Omega = (\lambda)$, then we say that the MRBA has weight $\lambda$.
\mlabel{it:mrb}
\end{enumerate}
\mlabel{de:rbr}
\end{defn}
Note that any Rota-Baxter algebra is also a MRBA by letting $\Omega$ contain a single element.

We first consider the simple case $K(x,t) = K(t)$, that is, $K$ is a \shadow kernel.  In this case $P_K$ is essentially the classical integral operator $I(f)(x):=\int_a^x f(t)dt$ and the Rota-Baxter identity of weight zero holds:
\begin{equation}
I(f)I(g)=I(fI(g))+I(I(f)g) \tforall f, g\in C(I).
\mlabel{eq:rbi}
\end{equation}
In fact, as $K\in C(I)$ varies, the family of operators $P_K$ form a MRBA in Definition~\mref{de:rbr}.(\mref{it:mrb}). More precisely, the following result holds.

\begin{prop}
	For any $K,H \in C(I)$, $P_K$ and $P_H$ satisfy the matching Rota-Baxter identity
	\[P_K(f)P_H (g)=P_K(fP_H(g)) + P_H(P_K(f)g) \tforall f, g\in C(I),\]
	making $(C(I),(P_K)_{K \in C(I)})$ a MRBA of weight $0$.
	\mlabel{th:mrb}
\end{prop}
\begin{proof}
	We note that $P_K(f)=I(Kf)$ and $P_H(g)=I(Hg)$ for the integral operator $I(f)(x):=\int_a^x f(t)dt$. Since $I$ is a Rota-Baxter operator of weight zero, we obtain
	$$
	P_K(f)P_H(g)= I(Kf)I(Hg)=I(KfI(Hg))+I(I(Kf)Hg)
	=P_K(fP_H(g))+P_H(P_K(f)g),
	$$
	as needed.
\end{proof}

We now consider the case when the kernel $K(x,t)$ is indeed a function of $x$ in addition to $t$.
Here, the Volterra integral operator $P_K$ is no longer a Rota-Baxter operator. As a simple example,
	let $K(x,t)=x$ and $f=g=1$. Then the linear operator $P_K$ on $C(\RR)$ with $a=0$ gives
	$P_K(f)(x)\,P_K(g)(x)=x^4$ which does not agree with
	$P_K(fP_K(g))(x)+P_K(P_K(f)g)(x)=
	\frac{2}{3}x^4.$

In general, a Volterra operator with separable kernel $K(x,t) = k(x)h(t)$ is a Rota-Baxter operator only in very specific circumstances, as shown in Corollary~\mref{co:septrb}.
Nevertheless, the operator can satisfy a twisted Rota-Baxter identity.

\begin{defn} \mlabel{de:trba}
\begin{enumerate}
\item
	For any $\lambda \in \bfk$, an   algebra $R$ with a linear operator $P$ and a specific invertible element $\mfraka\in R$ is called a {\bf \tw Rota-Baxter algebra of weight $\lambda$ with twist $\mfraka$} if
	\begin{equation}\mlabel{eq:trb}
	P(x)P(y)=\mfraka P\left(\mfraka^{-1}P(x)y\right)+\mfraka P\left(\mfraka^{-1}xP(y)\right)+\lambda \mfraka P\left(\mfraka^{-1}xy\right) \quad \tforall x, y\in R.
	\end{equation}
\item Fix a non-empty set $\Omega$ and let $\lambda_\Omega :=(\lambda_\omega \,|\, \omega \in \Omega) \subseteq \bfk$. An {\bf $\Omega$-matching twisted Rota-Baxter algebra of weight $\lambda_\Omega$ with twist $\mfraka_\Omega$}, or simply an {\bf $\Omega$-\smtrba}, is a $\bfk$-algebra $R$ together with a family of linear operators $P_\Omega := (P_\omega \,|\, \omega \in \Omega)$, where $P_\omega:R \to R$ for all $\omega \in \Omega$, and a parametric family $\mfraka_\Omega :=(\mfraka_\omega \,|\, \omega \in \Omega)$ of invertible elements of $R$ satisfying
	\begin{equation}\mlabel{eq:mtrb}
	P_\alpha(x)P_\beta(y)=\mfraka_\alpha P_\beta\left(\mfraka_\alpha^{-1}P_\alpha(x)y\right)+\mfraka_\beta P_\alpha\left(\mfraka_\beta^{-1}xP_\beta(y)\right)+ \lambda_\beta\mfraka_\beta P_\alpha\left(\mfraka_\beta^{-1}xy\right)
	\end{equation}
for all $x, y\in R,\alpha,\beta\in\Omega$. If $\mfraka_\Omega = (\mfraka)$ (resp. $\lambda_\Omega = (\lambda)$), then we say that the \mtrba has twist $\mfraka$ (resp. weight $\lambda$).
\end{enumerate}

	For $\Omega$-\mtrbas $(R,P_\Omega,\mfraka_\Omega)$ and $(R',P'_\Omega,\mfraka'_\Omega)$, an algebra homomorphism $\varphi:R\rar R'$ is called an {\bf $\Omega$-\mtrba homomorphism} if $\varphi(\mfraka_\omega)=\mfraka'_\omega$ and $\varphi P_\omega=P'_\omega \varphi$ for all $\omega\in\Omega$.
\end{defn}

The next result shows that a \mtrba is equivalent to a MRBA.
\begin{prop}
Let $P_\Omega$ be a family of linear operators on an algebra $R$ and let $\mfraka_\Omega\subseteq R$ of invertible elements and $\lambda_\Omega :=(\lambda_\omega \,|\, \omega \in \Omega) \subseteq \bfk$. Then $(R,P_\Omega)$ is an $\Omega$-\mtrba of weight $\lambda_\Omega$ with twist $\mfraka_\Omega$ if and only if $(R,\check{P}_\Omega)$ is an $\Omega$-MRBA of weight $\lambda_\Omega$, where $\check{P}_\omega:=\mfraka_\omega^{-1}P_\omega\mfraka_\omega,\,\omega\in\Omega$.
	\mlabel{pp:twrbo}
\end{prop}
\begin{proof}
	The proposition follows from the equation
	\begin{eqnarray*}
		&& P_\alpha(x)P_\beta(y)-\mfraka_\alpha P_\beta(\mfraka_\alpha^{-1}P_\alpha(x)y)-\mfraka_\beta P_\alpha(\mfraka_\beta^{-1}xP_\beta(y))-\lambda_\beta\mfraka_\beta P_\alpha(\mfraka_\beta^{-1}xy)\\
		&=& \hspace{-.3cm} \mfraka_\alpha\mfraka_\beta\Big(\check{P}_\alpha(\mfraka_\alpha^{-1}x)\check{P}_\beta(\mfraka_\beta^{-1}y)
-\check{P}_\beta\big(\check{P}_\alpha(\mfraka_\alpha^{-1}x)(\mfraka_\beta^{-1}y)\big)
-\check{P}_\alpha\big((\mfraka_\alpha^{-1}x)\check{P}_\beta(\mfraka_\beta^{-1}y)\big)
-\lambda_\beta\check{P}_\alpha((\mfraka_\alpha^{-1}x)(\mfraka_\beta^{-1}y))
\Big)
	\end{eqnarray*}
relating the axioms of the \mtrbas and the \mrbas.
\end{proof}

\noindent
{\bf Convention. } With our application to integral equations in mind, we will only consider the case of weight zero in the rest of the paper. So for an \mtrba, we always mean one with weight zero.

We give more identities for an \mtrba (of weight 0) for later use.
\begin{lemma}\mlabel{lem:trbr}
	Let $(R,P_\Omega,\mfraka_\Omega)$ be an $\Omega$-\mtrba.
	Then for $\check{P}_{\omega,\omega'}:=\mfraka_{\omega'}^{-1}P_\omega\mfraka_{\omega'},\,\omega,\omega'\in\Omega$, we have
	\begin{align}
\mlabel{eq:trbr}
&P_\alpha(x)\check{P}_{\beta,\gamma}(y)-\mfraka_\alpha \check{P}_{\beta,\gamma}(\mfraka_\alpha^{-1}P_\alpha(x)y)-\mfraka_\beta\check{P}_{\alpha,\gamma}(\mfraka_\beta^{-1}x\check{P}_{\beta,\gamma}(y))=0,\\
\mlabel{eq:mtrbr}
&\check{P}_{\alpha,\eta}(x)\check{P}_{\beta,\gamma}(y)-\mfraka_\alpha\mfraka_\eta^{-1}\check{P}_{\beta,\gamma}(\mfraka_\alpha^{-1}\mfraka_\eta\check{P}_{\alpha,\eta}(x)y)
-\mfraka_\beta\mfraka_\gamma^{-1}\check{P}_{\alpha,\eta}(\mfraka_\beta^{-1}\mfraka_\gamma x\check{P}_{\beta,\gamma}(y))=0,
	\end{align}
for any $x, y \in R$ and $\alpha,\beta,\eta,\gamma\in\Omega$.
\end{lemma}
\begin{proof}
To prove Eq.~\meqref{eq:trbr}, we multiply $\mfraka_\gamma$ to its left hand side and then use Eq.~\meqref{eq:mtrb} to derive
\begin{align*}
\mfraka_\gamma&P_\alpha(x)\check{P}_{\beta,\gamma}(y)-\mfraka_\gamma\mfraka_\alpha \check{P}_{\beta,\gamma}(\mfraka_\alpha^{-1}P_\alpha(x)y)-\mfraka_\gamma\mfraka_\beta\check{P}_{\alpha,\gamma}(\mfraka_\beta^{-1}x\check{P}_{\beta,\gamma}(y))\\
&=P_\alpha(x)P_\beta(\mfraka_\gamma y)-\mfraka_\alpha P_\beta(\mfraka_\alpha^{-1}P_\alpha(x)(\mfraka_\gamma y))
-\mfraka_\beta P_\alpha(\mfraka_\beta^{-1}x P_\beta(\mfraka_\gamma y))
=0.
\end{align*}

For Eq.~\meqref{eq:mtrbr}, we also use Eq.~\meqref{eq:mtrb} to see that
\begin{align*}
\check{P}_{\alpha,\eta}(x)&\check{P}_{\beta,\gamma}(y)-\mfraka_\alpha\mfraka_\eta^{-1}\check{P}_{\beta,\gamma}(\mfraka_\alpha^{-1}\mfraka_\eta\check{P}_{\alpha,\eta}(x)y)
-\mfraka_\beta\mfraka_\gamma^{-1}\check{P}_{\alpha,\eta}(\mfraka_\beta^{-1}\mfraka_\gamma x\check{P}_{\beta,\gamma}(y))\\
&=\mfraka_\eta^{-1}\mfraka_\gamma^{-1}\left(P_\alpha(\mfraka_\eta x)P_\beta(\mfraka_\gamma y)-\mfraka_\alpha P_\beta(\mfraka_\alpha^{-1}P_\alpha(\mfraka_\eta x)(\mfraka_\gamma y))- \mfraka_\beta P_\alpha(\mfraka_\beta^{-1}(\mfraka_\eta x)P_\beta(\mfraka_\gamma y))\right)=0. \qedhere
\end{align*}
\end{proof}

As the main application of \mtrbas, we have

\begin{theorem}
	Let $I\subseteq \RR$ be an open interval and let $R:=C(I)$ be the algebra of continuous functions on $I$. Let $K_\omega(x,t)=k_\omega(x)h_\omega(t)\in C(I^2),\,\omega\in\Omega$ be a family of separable kernels with all $k_\omega(x)$ free of zeros. Let $a\in I$ and let $P_{K_\omega}(f)(x):=\int_a^xK_\omega(x,t)f(t)\,dt, \,\omega\in\Omega,$ be the corresponding Volterra operators on $R$. Denote $\mfraka_\omega:=\frac{k_\omega(x)}{k_\omega(a)},\,\omega\in\Omega$. Then $(R,P_{K_\Omega},\mfraka_\Omega)$ is an $\Omega$-\mtrba.
	\mlabel{thm:twfunction}	
\end{theorem}

\begin{proof}
By assumption, the functions $\mfraka_\omega=\frac{k_\omega(x)}{k_\omega(a)}, \,\omega\in\Omega,$ are invertible on $I$. To verify that $(C(\RR),P_{K_\Omega},\mfraka_\Omega)$ is an $\Omega$-\mtrba, we just need to check Eq.~\meqref{eq:mtrb} as follows. For any $f,g\in C(I)$ and $\alpha,\beta\in\Omega$,
\begin{align*}
\mfraka_\alpha P_{K_\beta}&\left(\mfraka_\alpha^{-1}\big(P_{K_\alpha}(f)g\big)\right)+\mfraka_\beta P_{K_\alpha}\left(\mfraka_\beta^{-1}\big(fP_{K_\beta}(g)\big)\right)\\
	&=\frac{k_\alpha(x)}{k_\alpha(a)}
	\left(k_\beta(x)\int_a^x \frac{h_\beta(t)k_\alpha(a)}{k_\alpha(t)}\left(g(t)k_\alpha(t)\int_a^th_\alpha(u)f(u)\,du\right)\,dt\right)\\
    &\quad+\frac{k_\beta(x)}{k_\beta(a)}
	\left(k_\alpha(x)\int_a^x \frac{h_\alpha(t)k_\beta(a)}{k_\beta(t)}\left(f(t)k_\beta(t)\int_a^th_\beta(u)g(u)\,du\right)\,dt\right)\\
	&=k_\alpha(x)k_\beta(x)\left(\int_a^x h_\beta(t)g(t)\,dt\int_a^th_\alpha(u)f(u)\,du+\int_a^x h_\alpha(t)f(t)\,dt\int_a^th_\beta(u)g(u)\,du\right)\\	 &=k_\alpha(x)k_\beta(x)\left(\int_a^xh_\alpha(t)f(t)\,dt\right)\left(\int_a^xh_\beta(t)g(t)\,dt\right)=P_{K_\alpha}(f)P_{K_\beta}(g),
	\end{align*}
	where the third equality is due to integration by parts.
\end{proof}

\begin{coro}
With the notations in Theorem~\mref{thm:twfunction}, if $\Omega$ is a singleton, then $(C(I),P_{K_\Omega})$ is a \tw Rota-Baxter algebra with twist $\mfraka:=k(x)/k(a)$.
\end{coro}

\begin{coro}
	\mlabel{co:septrb}
With the notations in Theorem~\mref{thm:twfunction}, a Volterra operator $P_{K_\omega},\omega\in\Omega$ is a Rota-Baxter operator if $k_\omega(x)$ is a constant function, that is, the kernel $K_\omega$ is phantom. When all  kernels $K_\omega,\omega\in\Omega,$ are phantom, $(C(I),P_{K_\Omega})$ is a MRBA.
\end{coro}

\begin{proof}
According to Theorem~\mref{thm:twfunction}, $P_{K_\omega}$ is a \tw Rota-Baxter operator with $\mfraka_\omega=k_\omega(x)/k_\omega(a)$, which is a Rota-Baxter operator if $\mfraka_\omega=1$. This holds if and only if $k_\omega(x)=k_\omega(a)$ is a constant. In this Eq.~\meqref{eq:mtrb} becomes Eq.~\meqref{eq:mrb}.
\end{proof}

\begin{exam}
	\begin{enumerate}
		\item For $K(x,t)=e^{-x+t}=e^t/e^x$. The operator $P_K:C(\RR)\to C^1(\RR)$ satisfies the Reynolds identity in Eq.~\meqref{eq:rey2}. See~\mcite{Bi,RoR,Ro1} for further details on the identity.
		On the other hand, for $\mfraka=k(x)/k(a)=e^{a-x}$, we obtain a \tw Rota-Baxter operator.
		\mlabel{it:exp}
		\item For $K(x,t) = x$, taking $\mfraka=k(x)/k(a)=x/a$, we obtain a \tw Rota-Baxter operator.
	\end{enumerate}
	\mlabel{ex:sepex}
\end{exam}

\section{An algebraic framework for integral equations}
\mlabel{sec:gen}

We now introduce a framework to express integral equations, in particular the Fredholm equations and Volterra equations. It is our hope that this approach will on one hand provide a uniform and precise context for studying integral equations that are more general than their existing forms and, on the other hand, will also bring the algebraic methods and perspectives into the study of integral equations and lay the foundation for their symbolic computations. For some related literature, see~\mcite{BLRUV,GGR,GGZ,GRR,RR}.

\subsection{Background}
We will first express integral equations in terms of elements of a suitable operated algebra~\mcite{Gop,Gub}.
An {\em algebraic} integral equation can be understood as the annihilation of an ``integral" algebraic expression $\Phi(Y,P_\Omega,A)$ consisting of several ingredients and restrictions~\mcite{Ze}:

\begin{enumerate}
\item
an algebra of variable functions, including a set $Y$ of unknown functions to be solved from the integral equation;
\mlabel{it:opeq1}
\item
a set $P_\Omega:=(P_\omega\,|\,\omega\in\Omega)$ of integral operators in various forms, set apart by
\begin{enumerate}
\item the lower or upper limits (each being fixed or variable, and in the later case, independent variables or intermediate variables);
\item the kernels for the integral operators, as functions in both the dummy variables of the integrals and the independent variables for the integral equation;
\end{enumerate}
Special cases are the Volterra operators and the Fredholm operators;
\item
a set $A$ of ``free term" or coefficient functions which can appear both inside and outside of the integrals. Some of them can be constant, treated as parameters.
\item
These ingredients are put together by the algebraic operations together with the action of the integral operators.
\end{enumerate}

\begin{remark}
	\mlabel{rk:opalg}
We note a subtle difference in the meanings of algebraicity of an operator equation, as discussed in Item~\meqref{it:opeq1}. If the unknown functions are polynomials of the unknown variables $Y$, then the equation is called {\bf algebraic}, while if other functions of $Y$ can appear, the equation is called {\bf operator algebraic}. In the following integral equations, the first one is algebraic while the second one is operator linear since $y^{3/2}$ is not polynomial in $y$. But the second equation is still operator algebraic since the operations for the integrals are algebraic (and in composition).
\end{remark}

We put Example~\mref{ex:classic} into this framework.

\begin{exam}
\begin{enumerate}
\item {\bf Volterra's population model}. Here $Y = \{u\}$, there is one integral operator $P$ defined by Eq.~\meqref{eq:pop}
and the set $A$ can be taken to be the set of all constants.  Then Volterra's population model in Eq.~\meqref{eq:volt} becomes
\[
u = u_0 + aP(u) - bP(u^2) - cP(uP(u)).
\]
\item
{\bf Thomas-Fermi equation}. In this setting, $Y = \{y\}$, there are two Volterra operators $P_1$ and $P_2$ in Eq.~\meqref{eq:pot}
and the set $A$ can be taken to be the ring of polynomials $\RR[x]$.  Then the Thomas-Fermi equation \meqref{eq:tf} becomes
\begin{equation} \mlabel{eq:tf2}
y = 1 + Bx + P_1(P_2(y^{3/2})).
\end{equation}
\end{enumerate}
\end{exam}

To inspire a precise general framework for integral equations, let us first recall how we formulate {\em algebraic} equations. An algebraic equation consists of several ingredients:

\begin{enumerate}
\item
a set $X$ of variables;
\item
a set of ``free term" elements from a prefixed $\bfk$-algebra $A$ over the ground field $\bfk$.
\item
The ingredients can be put together by the algebraic operations.
\end{enumerate}

As is well-known,
the general form of an algebraic equation is an element in the polynomial algebra $A[X]$.

Another motivation which is more closely related to our study is differential equations, consisting of ingredients and conditions:

\begin{enumerate}
\item
a set $Y$ of unknown functions;
\item
a set $d_\Omega:=(d_\omega\,|\, \omega\in \Omega)$ of differential operators;
\item
a set of ``free term" or coefficient functions from a differential algebra $(A,d_{0,\Omega})$ where $d_{0,\Omega}:=(d_{0,\omega}\,|\,\omega\in \Omega)$.
\item
The ingredients can be put together by the algebraic operations, together with the differential operators.
\end{enumerate}

Note that here $A$ carries its own commuting derivations $d_{0,\omega}, \omega \in \Omega$, called an {\bf $\Omega$-ring} in differential algebra~\mcite{Ko,Si} where $\Delta$ instead of $\Omega$ is used. Elements in $A$ are called the {\bf coefficient functions} because they are the coefficients in the differential equation. A coefficient function is not necessarily a constant function (defined to be the ones in $\bfk$) for the derivations $d_{0,\omega}$ or $d_\omega$.

In differential algebra~\mcite{Ko,Si}, a differential equation with coefficients in an $\Omega$-ring $(A,d_{0,\Omega})$ is an element in the differential polynomial algebra $A\{Y\}:=A_\Omega\{Y\}$ which is simply the polynomial algebra
\begin{equation} \mlabel{eq:diffpoly}
A_\Omega\{Y\}:=A[\Delta_\Omega Y], \text{ where }
 \Delta_\Omega Y:= Y\times C(\Omega).
\end{equation}
Here $C(\Omega)$ is the commutative monoid over $\Omega$:
$$C(\Omega):=\left\{\left .\prod_{\omega\in\Omega} \omega^{n_\omega}\,\right|\, n_\omega\in \NN, \omega\in \Omega, \text{ with finitely many } n_\omega \text{ positive}\right\}$$
representing the iterations from $\Omega$. An element $\mfraka=\prod_{\omega\in \Omega} \omega^{n_\omega}$ of $C(\Omega)$ is identified with the map
$$\bfn=\bfn_\mfraka:\Omega\to \NN, \quad \omega\mapsto n_\omega,$$
with finite support. For $\omega\in \Omega$, the derivation $d_\omega$ on $A_\Omega\{Y\}$ first restricts to $d_{0,\omega}$ on $A$ and further sends $(y,\mfraka)$ to $(y,\omega\mfraka)$, that is, sending $\bfn_\mfraka$ to $$\bfn_{\mfraka,\omega}: \Omega\to \NN, \alpha\mapsto \left\{\begin{array}{ll} \bfn(\alpha)+1, & \alpha=\omega, \\
\bfn(\alpha), & \text{ otherwise. } \end{array} \right .$$
These two conditions determine $d_\omega$ uniquely since $A[\Delta_\Omega Y]=A\ot_\bfk \bfk[\Delta_\Omega Y]$ as an algebra for which $d$ applies to a pure tensor $a\ot f=af$ by the Leibniz rule.  See~\mcite{Ko,PS,Ri} for more details.

For example, with $Y=\{y\}$ and $\Omega=\{d\}$ both singletons, $A\{Y\}=A\{y\}=A[y^{(n)}\,|\, n\geq 0]$ is the space for differential equations with one derivation $d$ and one unknown function $y$.

While this gives precise definitions of algebraic equations and (algebraic) differential equations,
there does not appear to be a general notion of integral equations, even though in broad terms an integral equation is an equation involving integral operators.

Based on the previous discussion, as an analogy of an algebraic or a differential equation,  an integral equation should be an element of a suitably defined ``integral polynomial algebra" which should be a free object in a suitable category of integral algebras.
Such integral algebras would depend on the integral operators involved. Since an arbitrary integral operator is not known to satisfy any algebraic relations,
in order to serve the general purpose, we impose only linearity on the operators, which leads us to an operated algebra~\cite{Gop}.

First in Section~\mref{ss:freeop} we give two constructions of free commutative operated algebras, one by bracketed words and one by typed decorated trees, that is, rooted trees whose vertices and edges are both decorated. This setup will provide a general framework for operator equations and, in particular, integral equations, with constant coefficients.

In order to deal with integral equations with variable coefficients, we consider a relative version of operated algebras in Section~\mref{ss:relfreeop}, in the sense of operated algebras with an operated homomorphism from a given operated algebra, in practice a ``coefficient algebra" of functions equipped with integral operators. The free objects in this category of relative operated algebras provide the formal and rigorous context for operator equations with variable coefficients.

Of course such an interpretation is not our goal by itself, but rather the first step in understanding integral equations from an algebraic viewpoint. Applications will be given in Section~\mref{sec:sep}.

\subsection{Free operated algebras and decorated rooted trees}
\mlabel{ss:freeop}

We construct free operated algebras first by bracketed works and then by rooted trees with decorations on both the vertices and the edges. \subsubsection{Construction of free operated algebras by bracketed elements}
We start with some general definitions and notations.

\begin{defn} \label{de:mapset}\mlabel{de:opalg}
	Let $\Omega$ be a nonempty set. An {\bf $\Omega$-operated $\bfk$-module} (resp. {\bf $\Omega$-operated $\bfk$-algebra}) is a $\bfk$-module (resp. $\bfk$-algebra) $V$ together with $\bfk$-linear maps $P_\omega: V\to V,  \omega\in \Omega$. Denote $P_\Omega:=(P_\omega\,|\, \omega\in\Omega)$.
	A homomorphism from an $\Omega$-operated $\bfk$-module (resp. $\bfk$-algebra) $(V, P_\Omega)$ to an $\Omega$-operated $\bfk$-module (resp. $\bfk$-algebra) $(U,Q_\Omega)$ is a $\bfk$-module (resp. $\bfk$-algebra) homomorphism $f :V\to U$ such that $f P_\omega= Q_\omega f, \omega\in \Omega$.
\end{defn}

The class of $\Omega$-operated   algebras with homomorphisms between them forms the category of $\Omega$-operated   algebras.

Let $C$ be an   algebra.
We construct the free $\Omega$-operated unitary algebra over $C$, by applying bracketed words. See~\mcite{Gop} for more details on bracketed words.

\begin{defn}
	\mlabel{de:freeabalg}
	Let $\Omega$ be a nonempty set and $C$ a $\bfk$-algebra. The {\bf free $\Omega$-operated $\bfk$-algebra} over $C$ is an $\Omega$-operated $\bfk$-algebra $(\fopaa,\Pi_\Omega)$ together with a $\bfk$-algebra homomorphism $i_C:C\lra \fopaa$ satisfying the universal property that, for any $\Omega$-operated $\bfk$-algebra $(R,P_\Omega)$ and $\bfk$-algebra homomorphism $f:C\to R$, there is a unique $\Omega$-operated $\bfk$-algebra homomorphism $\free{f}: (\fopaa,\Pi_\Omega) \to (R,P_\Omega)$ such that $\free{f} i_C=f$.
\end{defn}

In the special case when $C$ is the symmetric algebra $S(V)$ on a module $V$ (resp. the polynomial algebra $\bk[Y]$ on a set $Y$), we obtain the free $\Omega$-operated algebra over the module $V$ (resp. over the set $Y$), where the algebra homomorphism $i_C$ is replaced by a linear map (resp. a set map). Recall $S(V) = \bigoplus_{n\ge0}S^n(V)$ where $S^n(V)$ is the $n$-th symmetric tensor power of $V$ with the convention that $S^0(V) = \bfk$.  Denote $S^+(V) := \bigoplus_{n\ge1} S^n(V)$.

Given an $\omega \in \Omega$ and a module $V$, let $\lfloor V \rfloor_{\omega}$ denote the module $\big\{\lfloor v\rfloor_{\omega}\mid v\in V\big\}$. So it is a copy of $V$ but is linearly independent of $V$. We also assume that the modules $\lfloor V \rfloor_{\omega}$ are linearly independent when $\omega$ varies in $\Omega$. Another way to think of $\lc V\rc_\omega$ is that $\lc V\rc_\omega$ is a module together with a given linear bijection $\pi_\omega: V\to \lc V\rc_\omega$ for which we denote $\pi_\omega (v)=\lc v\rc_\omega$, hence the notation $\lc V\rc_\omega$. In reality, $\lc v\rc_\omega$ is meant to be the image of $v$ under the action of a linear operator $\lc\cdot \rc_\omega$. Note that by definition, we have the linearity
$$ \lc au+bv\rc_\omega = a\lc u\rc_\omega + b\lc v\rc_\omega \tforall a, b\in \bfk, u, v\in V.$$
Denote
$$\lc V \rc_\Omega := \bigoplus_{\omega \in \Omega} \lc V \rc_\omega.$$

With the notations in Definition~\mref{de:freeabalg}, we now construct the free $\Omega$-operated algebra over an algebra $C$. The construction is modified after the one for free $\Omega$-operated algebras in~\mcite{Gop}, with the commutativity condition further imposed, as well as the condition on the generating set being relaxed to a generating algebra. See the remark after Definition~\mref{de:freeabalg}.

We will obtain the free object as the limit of a direct system
\begin{align*}
\{i_{n,\,n+1}: \algw(C)_n\rightarrow \algw(C)_{n+1}\}_{n=0}^{\infty}
\end{align*}
of algebras $\algw_n:=\algw(C)_{n},n\geq 0$, where the transition morphisms $i_{n,\,n+1}$ are natural injective algebra homomorphisms.

For the initial step of $n=0$, we define
$$\algw_0:=C.$$
We then define
\begin{equation*}
\algw_1:=\algw(C)_1:= C\ot S(\lc C \rc_\Omega)=C\oplus (C\ot S^+(\lc C\rc_\Omega)),
\end{equation*}
with the natural injective algebra homomorphism
$$ i_{0,\,1}: \algw_0:=C\hra \algw_1=C\oplus (C\ot S^+(\lc C\rc_\Omega)).$$

For the inductive step, let $n\geq 1$ be given and assume that the algebras $\algw_k, k\leq n,$ and injective algebra homomorphisms
\begin{equation}
i_{k-1,\,k}: \algw_{k-1}\hra \algw_k, 1\leq k\leq n,
\mlabel{eq:inclu2ab}
\end{equation}
have been defined.
We then define
\begin{equation}
\begin{aligned}
\algw_{n+1}&:=\algw(C)_{n+1}:=C\ot S(\lc \algw_n\rc_\Omega) =C\oplus (C\ot S^+(\lc \algw_n\rc_\Omega)).
\end{aligned}
\mlabel{eq:wrecab}
\end{equation}

The natural injection $i_{n-1,\,n}:\algw_{n-1}\hra \algw_n$ in Eq.~(\mref{eq:inclu2ab}) induces the natural linear injection
$$\lc \algw_{n-1} \rc_\Omega \hookrightarrow \lc \algw_n \rc_\Omega$$
and thus an algebra injection
$$ i_{n,\,n+1}:\algw_n
:=C\ot S\big(\lc \algw_{n-1}\rc_\Omega\big)
\hookrightarrow C\ot S\big(\lc \algw_n \rc_\Omega\big) =:\algw_{n+1}.$$

This completes the inductive construction of the direct system. Finally we define the direct limit of algebras
\begin{equation} \mlabel{eq:dirw}
\algw(\Omega,C):=\lim_{\longrightarrow}\algw_{n}=\bigcup_{n\geq 0}\algw_{n}.
\end{equation}
As a direct limit of algebras, $\algw(\Omega,C)$ is an algebra.

For $\omega\in \Omega$, define a linear map
$$
\opw_{\omega}:\algw(\Omega,C)\rightarrow \algw(\Omega,C),\, u\mapsto
\lc u \rc_{\omega}.
$$

Thus the pair $\big(\algw(\Omega,C), \opw_\Omega\big)$ is an $\Omega$-operated algebra.
Adapting the proof of the free $\Omega$-operated algebras in~\cite[Corollary~3.6]{Gop}, we can show that $\big(\algw(\Omega,C), \opw_\Omega\big)$ is the free   $\Omega$-operated algebra over $C$. The proof is also a simplified version of the proof for the relative case in Theorem~\mref{thm:freea}. So the reader can also be referred there for details.

\begin{theorem}
\mlabel{thm:freeaab}
Let $\Omega$ be a nonempty set and $C$ a   $\bfk$-algebra. Let $j_{C}: C\hra \algw(\Omega,C)$ be the natural injection.
The pair $(\algw(\Omega,C),\, \opw_\Omega)$ together with $j_C$ is the free  $\Omega$-operated $\bfk$-algebra over $C$.
\end{theorem}

\subsubsection{Realization as typed trees}

Now we apply typed rooted trees, with decorated vertices and edges, to give another construction of the free  $\Omega$-operated algebra $\algw(\Omega,C)$ just obtained by bracketed words. The rooted tree construction has the advantage of being intuitive and non-recursive, giving an easy understanding of the free object. The bracketed word construction is more precise, and the recursive definition is convenient to be applied in proofs.

See~\mcite{BHZ,Foi} for more details on typed rooted trees from the perspective of renormalization, algebra and combinatorics. See also~\mcite{CGPZ} for a related treatment for the rooted tree integrals and sums, and their renormalization.

Let $\sett$ denote the set of nonplanar rooted trees.
For a nonempty set $\Omega$ and an algebra $C$, let $\sett(\Omega,C)$ denote the set of the trees in $\sett$ with their vertices decorated by elements of $C$ and their edges decorated by elements of $\Omega$. Such trees are called {\bf vertex-edge decorated trees}. Some examples are
\vspace{-.7cm}
$$
\ctone{a}\quad
\cttwo{a}{b}{\alpha}\quad
\ctthree{a}{b}{c}{\alpha}{\beta}\quad
\ctfour{a}{b}{c}{d}{\alpha}{\beta}{\gamma}\quad
\ctfourb{a}{b}{c}{d}{\alpha}{\beta}{\gamma}\quad
\ctfourc{a}{b}{c}{d}{\alpha}{\beta}{\gamma}\quad
\ctfive{a}{b}{c}{d}{e}{\alpha}{\beta}{\gamma}{\delta} \!\!\!\! , a, b, c, d, e\in C, \alpha, \beta, \gamma, \delta\in \Omega.$$

Let $\algt(\Omega,C)$ denote the module spanned by $\sett(\Omega,C)$, allowing ($\bk$-)linearity on decorations of the vertices.

We next define the {\bf \graftprod} $T \tprod U$ of two $(\Omega,C)$-decorated trees $T$ and $U$ to be obtained by merging the roots of $T$ and $U$ into a common root shared by the branches of both the trees $T$ and $U$, and decorating the common root by the product of the decorations of the roots of $T$ and $U$.
For example
$$ \ctone{a}\tprod \quad \ctfour{b}{c}{d}{e}{\alpha}{\beta}{\gamma} = \quad \ctfour{ab}{c}{d}{e}{\alpha}{\beta}{\gamma}, \quad
\cttwo{a}{b}{\alpha} \tprod\quad  \cttwo{c}{d}{\beta} =\ \ \  \ctthree{ac}{b}{d}{\alpha}{\beta}.$$

For each $\omega\in\Omega$, define a linear operator $\opt_\omega$ on $\algt(\Omega,C)$ to be the {\bf \extnop}, sending a decorated tree $T$ in $\sett(\Omega,C)$ to a new tree $\opt_\omega(T)$ by adding a new root connecting to the root of $T$, decorating the edge connecting the new root to the root of $T$ by $\omega$ and decorating the new root by $1$.
For example,
$$\opt_{\omega}(\ctone{a})=\cttwo{1}{a}{\omega}, \quad \opt_{\omega}\Big( \hspace{.4cm} \ctthree{a}{b}{c}{\alpha}{\beta}\Big) =\quad \ctfourc{1}{a}{b}{c}{\omega}{\alpha}{\beta}.$$

Note that often in the literature the product of trees is given by concatenation and the unary operator is given by grafting. See for example Connes-Kreimer Hopf algebra of rooted trees \mcite{CK} and the constrution of free {\em noncommutative} operated algebras~\mcite{Gop}. Here the role of the two operations is somewhat reversed. Thus we use the terms \graftprod and \extnop to emphasize the difference.

With the notions above, the pair $((\algt(\Omega,C), \tprod), \opt_\Omega)$ is an $\Omega$-operated algebra.

We next define a bijection
$$\eta:\algw(\Omega,C) \to \algt(\Omega,C)$$
by the universal property of $\algw(\Omega,C)$ as the free $\Omega$-operated algebra over $C$.  More specifically, $\eta$ is defined as a direct limit of maps $\eta_n$ on the direct system $\algw_n, n\geq 0$, defining $\algw(\Omega,C)$ in Eq.~\meqref{eq:dirw}.  By showing that $\eta$ is a bijection, we will establish $\algt(\Omega,C)$ as another construction of the free $\Omega$-operated algebra on $C$.

In fact, $\algt(\Omega,C)$ is also the direct limit of a direct system $\algt_n, n\geq 0$, which equips $\algt(\Omega,C)$ with a filtered algebra structure. First define the {\bf height} of a rooted tree to be the length of the longest path from the root to the leafs. For $n\geq 0$, let $\sett_n$ denote the subset of $\sett$ with height less than or equal to $n$ and let $\algt_n$ be the linear span of $\sett_n$. From the definition of the product $\tprod$, $\algt_n$ is a subalgebra. Further the operator $\opt_\omega, \omega\in \Omega,$ sends $\algt_n$ to $\algt_{n+1}$.

To obtain the bijection $\eta:\algw(\Omega,C)\to \algt(\Omega,C)$, we first define
$$\eta_0: \algw_0=C\to \algt_0, \quad a \mapsto \ctone{a}, \quad a\in C,$$
evidently a bijective algebra homomorphism.

Next for a given $n\geq 0$, assume that a bijective algebra homomorphism
$\eta_n:\algw_n \to \algt_n$ has been defined. Then the bijection defines a linear bijection
$$\lc \eta_n\rc_\omega: \lc \algw_n\rc_\omega \to \opt_\omega(\algt_n), \quad
\lc u\rc_\omega \mapsto \opt_\omega(\eta_n(u)), \quad u\in \algw_n$$
and hence an algebra homomorphism
$$ \eta_{n+1}: \algw_{n+1}=C\ot S(\lc \algw_n\rc_\Omega)
\to \algt_{n+1}.$$
To be specific, consider the decomposition
$\algw_{n+1}=C \oplus (C\ot S^+(\lc \algw_n\rc_\Omega))$ and consider the two cases when $u$ is in either of the direct summands. If $u$ is in $C$, then define $\eta_{n+1}(u):=\ctone{u}.$ If $u$ is a pure tensor in $C\ot S^+(\lc \algw_n\rc_\Omega)$, then
$$u=c\ot (\lc u_1\rc_{\omega_1} \cdots \lc u_k\rc_{\omega_k}), \quad c\in C, u_i\in \algw_n, \omega_i\in \Omega, 1\leq i\leq k, k\geq 1,$$
with the multiplication in $S^+(\lc \algw_n\rc_\Omega)$ suppressed. Then $\eta_n(u_i), 1\leq i\leq k,$ are trees in $\sett_n$ by the inductive hypothesis. We thus define
\begin{equation}
 \eta_{n+1}(u):= \ctone{c}\tprod \opt_{\omega_1}(\eta_n(u_1))\tprod \cdots \tprod \opt_{\omega_k}(\eta_n(u_k)).
\mlabel{eq:phidef}
\end{equation}
By the definitions of the \extnop $\opt_{\omega_i}$ and the \graftprod $\tprod$,
the new tree $\eta_{n+1}(u)$ is the grafting of the trees $\opt_{\omega_i}(\eta_n(u_i)), 1\leq i\leq k,$ with $c$ decorating the common root. On the other hand, any rooted tree in $\sett_{n+1}$ is either $\ctone{c}$ or of the forms described above and thus has the factorization in Eq.~\eqref{eq:phidef}. Thus $\eta_{n+1}$ is surjective and an inverse map of $\eta_{n+1}$ can be defined by reversing the above process. This completes the inductive construction of $\eta$ and its bijectivity.

To summarize, we have
\begin{theorem}
The $\Omega$-operated $\bfk$-algebra $((\algt(\Omega,C),\tprod),\opt_\Omega)$, together with natural imbedding $i_C: C\to \algt(\Omega,C)$, is the free $\Omega$-operated $\bfk$-algebra over $C$.
\mlabel{thm:freeobtr}
\end{theorem}

\subsection{Free operated algebras in the relative context}
\mlabel{ss:relfreeop}

\subsubsection{General discussion}
The notion of an (associative) algebra is a relative notion of a ring in the sense that an algebra is a ring $R$ together with a ring homomorphism $\bk\to R$, called the structure map. A similar notion is fundamental in differential algebra where a differential algebra usually means a differential ring $(R,D)$ together with a differential ring homomorphism $(A,d)\to (R,D)$ from a base differential ring $(A,d)$.

Since an integral equation also has its coefficient functions equipped with preassigned integral operators, we give a relative notion of operated algebras.

\begin{defn}
For a given $\Omega$-operated $\bfk$-algebra $(A,\mrho_\Omega)$, an {\bf $(A,\mrho_\Omega)$-algebra} is an $\Omega$-operated $\bfk$-algebra $(R,P_\Omega)$ together with a homomorphism (called the {\bf structure map}) $j_R: (A,\mrho_\Omega) \lra (R,P_\Omega)$ of $\Omega$-operated $\bfk$-algebras.
A homomorphism $f:(R,P_\Omega)\to (S,Q_\Omega)$ of $(A,\mrho_\Omega)$-algebras is a homomorphism of $\Omega$-operated $\bfk$-algebras that is compatible with the structure maps:
$f j_R=j_S$.
\end{defn}

We next construct the free objects in the category of $(A,\mrho_\Omega)$-algebras.  Let $B$ be a $\bfk$-augmented algebra in the sense that there is a direct sum decomposition $$B=\bfk\oplus B^+$$ of $\bfk$-subalgebras. In other words, $B$ is the unitization of a $\bfk$-algebra $B^+$.
Then
$$ C=A\oplus (A\ot B^+)$$
is an {\bf $A$-augmented algebra} in the sense that there is a direct sum decomposition $C=A\oplus C_A^+$.

\begin{defn}
Let $\Omega$ be a nonempty set. Let $(A,\mrho_\Omega)$ be an  $\Omega$-operated $\bfk$-algebra and let $B$ be a $\bfk$-augmented algebra. The {\bf free $\Omega$-operated $\bfk$-algebra} over $B$ with coefficients in $(A,\mrho_\Omega)$ (or simply the {\bf free $(A,\mrho_\Omega)$-algebra} over $B$) is an $(A,\mrho_\Omega)$-algebra $(\fopab,\ropw_\Omega)$ together with an algebra homomorphism $i_B:B\lra \fopab$ satisfying the universal property that, for any $(A,\mrho_\Omega)$-algebra $(R,P_\Omega)$ and algebra homomorphism $f:B\to R$, there is a unique $(A,\mrho_\Omega)$-algebra homomorphism $\free{f}: (\fopab,\ropw_\Omega) \to (R,P_\Omega)$ such that $\free{f}\, i_B=f$.
\mlabel{de:freealg}
\end{defn}

\begin{remark} \mlabel{rk:opalg2}
As our primary interest, taking $B$ to be $\bfk[Y]$ for a set $Y$ of variable functions and replacing the algebra homomorphism from $B$ by a map from $Y$, we obtain the free $(A,\mrho_\Omega)$-algebras over the set $Y$.
However the more general $B$ allows non-algebraic functions in the variable functions to be included as noted in Remark~\mref{rk:opalg}.
See also Remark~\mref{rk:opalg3}.
\end{remark}

\subsubsection{Construction by bracketed words}

To construct the free $(A,\mrho_\Omega)$-algebra over a $\bfk$-augmented algebra $B$, we  denote
$$C:=A\ot B$$
and define a direct subsystem $\{\ralgw_n\,|\, n\geq 0\}$ of $\algw_n=\{\algw(C)_n\,|\, n\geq 0\}$. Each $\ralgw_n$ is an $A$-augmented algebra with $\ralgw_n=A\oplus \ralgw^+_n$.

First take
$$\ralgw_0:=\algw_0=C, \quad \ralgw^+_0:=C^+_A:=A\ot B^+,$$
and
$$\ralgw_1:=C\ot S(\lc \ralgw^+_0\rc_\Omega)=C\oplus (C\ot S^+(\lc \ralgw^+_0\rc_\Omega)) =A\oplus C^+_A \oplus (C\ot S^+(\lc \ralgw^+_0\rc_\Omega)),$$
$$
 \ralgw^+_1:=C^+_A \oplus (C\ot S^+(\lc \ralgw^+_0\rc_\Omega)).$$
Then we have the natural inclusions
$$\ralgw_0:=C\hra \ralgw_1, \quad \ralgw^+_0\hra \ralgw^+_1.$$

For the inductive step, let $n\geq 1$ be given and assume the data of algebras $\ralgw_k, \ralgw^+_k, 1\leq k\leq n,$ with $\ralgw_k=A\oplus \ralgw^+_k$, and injective algebra homomorphisms

\begin{equation}
i_{n-1,n}:\ralgw_{n-1}\to \ralgw_n, \quad i^+_{n-1,n}:\ralgw^+_{n-1}\to \ralgw^+_n.
\mlabel{eq:inclu2}
\end{equation}

We then define
\begin{equation}
\begin{aligned}
\ralgw_{n+1}&:=C\ot S(\lc \ralgw^+_n\rc_\Omega)
=A\oplus C^+_A\oplus (C\ot S^+(\lc \ralgw^+_n\rc_\Omega)), \\
\ralgw^+_{n+1}&:=C^+_A\oplus (C\ot S^+(\lc \ralgw^+_n\rc_\Omega)).
\end{aligned}
\mlabel{eq:wrec}
\end{equation}

The natural injection $i^+_{n-1,\,n}:\ralgw^+_{n-1}\hra \ralgw^+_n$ in Eq.~(\mref{eq:inclu2}) induces the natural injections
$$\lc \ralgw^+_{n-1} \rc_\Omega \hookrightarrow \lc \ralgw^+_n \rc_\Omega$$
and
$$ S(\lc \ralgw^+_{n-1}\rc_\Omega) \hookrightarrow S(\lc \ralgw^+_n\rc_\Omega),$$
which then induce
$$ i_{n,\,n+1}:\ralgw_n:=C\ot S(\lc \ralgw^+_{n-1}\rc_\Omega)  \hookrightarrow C\ot S\big(\lc \ralgw^+_n \rc_\Omega\big)=:\ralgw_{n+1},$$
$$ i^+_{n,\,n+1}:\ralgw^+_n:=C^+_A\oplus \big(C\ot S^+( \lc \ralgw^+_{n-1} \rc_\Omega)\big) \hookrightarrow C^+_A\oplus \big(C\ot S^+( \lc \ralgw^+_n \rc_\Omega)\big)=:\ralgw^+_{n+1}.$$
This completes the inductive construction of the direct systems. Finally we define the direct limits of algebras
$$ \ralgw :=\ralgw(\Omega,A\ot B):=\lim_{\longrightarrow}\ralgw_{n}=\bigcup_{n\geq 0}\ralgw_{n},\quad \ralgw^+:=\lim_{\longrightarrow}\ralgw^+_{n}=\bigcup_{n\geq 0}\ralgw^+_{n}.$$
As a direct limit of $\bk$-algebras, $\ralgw$ and $\ralgw^+$ are $\bk$-algebras.

For each $\omega\in \Omega$, $\lc \ralgw^+_n\rc_\omega$ is contained in $\ralgw^+_{n+1}$. Also, taking direct limits on the two sides in the equality $\ralgw_n=A\oplus \ralgw_n^+$, we obtain the equality
$\ralgw=A\oplus \ralgw^+$. Thus there is a natural $\bfk$-linear map
\begin{align*}
\ropw_{\omega}:\ralgw(\Omega,A\ot B)\rightarrow \ralgw(\Omega,A\ot B),\, u\mapsto \left\{\begin{array}{ll} \mrho_\Omega(u), &u\in A,\\
\lc u \rc_{\omega}, & u\in \ralgw^+, \end{array}
\right .
\, \omega\in \Omega.
\end{align*}
Consequently the pair $\big(\ralgw(\Omega,A\ot B), \ropw_\Omega\big)$ is an $\Omega$-operated algebra.
Further, there is the obvious inclusion $j_\ralgw:(A,\mrho_\Omega)\hra (\ralgw(\Omega,A\ot B),\ropw_\Omega)$ of $\Omega$-operated algebras. Hence $(\ralgw(\Omega,A\ot B),\ropw_\Omega)$ is an $(A,\mrho_\Omega)$-algebra.

We prove that $\ralgw(\Omega,A\ot B)$ satisfies the desired universal property.

\begin{theorem}
Let $\Omega$ be a nonempty set and let $(A,\mrho_\Omega)$ be an $\Omega$-operated algebra. Let $B$ be a $\bfk$-augmented algebra and let $i_{B}: B\hra \ralgw(\Omega,A\ot B)$ be the natural injection.
The pair $(\ralgw(\Omega,A\ot B),\, \ropw_\Omega)$ together with $i_B$ is the free $(A,\mrho_\Omega)$-algebra over $B$ in the sense of Definition~\mref{de:freealg}.
\mlabel{thm:freea}
\end{theorem}

\begin{proof}
Let $(R,P_\Omega)$ be an $(A,\mrho_\Omega)$-algebra with the structure map $j_R: (A,\mrho_\Omega)\to (R,P_\Omega)$ which is an $\Omega$-operated algebra homomorphism. Let $f:B\to R$ be an algebra homomorphism.
We will define the desired $(A,\mrho_\Omega)$-algebra homomorphism
$$\free{f}: (\ralgw(\Omega,A\ot B),\ropw_\Omega) \to (R,P_\Omega)$$
from a sequence of algebra homomorphisms
$$
f_n: \ralgw_n \to R, \quad n\geq 0,$$
that is compatible with the direct system $\{\ralgw_n,i_{n,n+1}\}_{n\geq 0}$.

Again taking $C:=A\ot B$, first define
$$ f_0=j_R\ot f: \ralgw_0:=C=A\ot B \to R.$$
Then define
$$ \lc f_0\rc: \lc \ralgw^+_0\rc_\omega =\lc C^+_A\rc_\omega \to R$$
by $\lc f_0\rc (\lc u\rc_\omega):= P_\omega (f_0(u))$, which then together with $f_0:C\to R$ induces a unique algebra homomorphism
$$ f_1: \ralgw_1=C\ot S(\lc \ralgw^+_0\rc_\Omega)\to R.$$
We also have $f_1\Big|_{\ralgw_0}=f_0$.

Inductively, for $n\geq 1$, suppose that $f_i: \ralgw_i\to R, i\leq n,$ have been defined such that $f_i\Big|_{\ralgw_{i-1}}=f_{i-1}, i\geq 1$.
We then define
$$\lc f_n\rc: \lc \ralgw^+_n\rc_\Omega \to R, \quad \lc u\rc_\omega \mapsto P_\omega(f_n(u)), \quad u\in \ralgw^+_n, \omega\in \Omega.$$
Then together with the map $f_0:C\to R$, we obtain an algebra homomorphism
$$ f_{n+1}: \ralgw_{n+1}=C\ot S(\lc \ralgw^+_n\rc_\Omega) \to R.$$
For $u\in \ralgw^+_{n-1}$, we have
$$ f_{n+1}(\lc u\rc_\omega)=\lc f_n\rc (\lc u\rc_\omega) = P_\omega (f_n(u))=P_\omega (f_{n-1}(u))=\lc f_{n-1}\rc(\lc u\rc_\omega)=f_n(\lc u\rc_\omega).$$
It then follows that $f_{n+1}\Big|_{\ralgw_n}=f_n$.
Therefore, we obtain an algebra homomorphism
$$ \free{f}:= \dirlim f_n: \ralgw(\Omega,A\ot B) =\dirlim \ralgw_n \to R.$$

To check that $\free{f}$ is an $(A,\mrho_\Omega)$-algebra homomorphism, for $a\in A$, we have
$$ \free{f}j_\ralgw(a)=\free{f}(a)=f_0(a)=j_R(a).$$
To see that $\free{f}$ is an operated algebra homomorphism, for $u\in \ralgw^+= \bigcup_{n\geq 0} \ralgw^+_n$ and hence $u\in \ralgw^+_n$ for some $n\geq 0$, by the definition of $\free{f}=\dirlim f_n$, we have
$$ \free{f} \ropw_\omega(u)=\free{f}(\lc u\rc_\omega) = f_{n+1}(\lc u\rc_\omega) = \lc f_n\rc (\lc u\rc_\omega) = P_\omega f_n(u) =P_\omega \free{f}(u).$$
Further, for $a\in A$, we have
$$\free{f}\ropw_\omega(a)=\free{f}(\mrho_\Omega(a)) =j_R(\mrho_\Omega(a))=P_\omega j_R(a)=P_\omega \free{f} j_\ralgw(a)
=P_\omega \free{f}(a).$$
This proves the existence of $\free{f}$.

Finally, from the construction of $\free{f}$, it is the only way to define an $(A,\mrho_\Omega)$-algebra homomorphism such that $\free{f}  i_B = f$. This completes the proof.
\end{proof}
\subsubsection{Construction by rooted nonplanar forests}

Again take $C:=A\ot B$ for a $\bfk$-augmented algebra $B$, in particular $B=\bfk[Y]$ for a set $Y$.
We consider the submodule $\ralgt(\Omega,C)$ of $\algt(\Omega,C)$ spanned by vertex-edge decorated rooted trees $\rsett(\Omega,C)\subseteq \sett(\Omega,C)$ for which the leaf vertices are decorated by $C^+_A:=A\ot B^+$. For the one-vertex tree, the vertex is not taken to be a leaf and hence can be decorated by any element in $C$. Then $\rsett(\Omega,C)=\{\ctone{a}\,|\,a\in A\} \sqcup \rsett^+$, where $\rsett^+$ consists of $\ctone{a}$ with $a\in C^+_A$ and trees of height at least one with leaf vertices decorated by $C^+_A$. Then $\ralgt(\Omega,C)$ is closed under the \graftprod $\tprod$. Further, for $\omega\in \Omega$, define $\ropt_\omega$ on $\rsett(\Omega,C)$ by
$$ \ropt_\omega(T)=\left\{\begin{array}{ll}
\ctone{\mrho_\Omega(a)}\ \ \ \ \,  , & T=\ctone{a}, a\in A,\\
\opt_\omega(T), & T\in \rsett^+(\Omega,C).
\end{array} \right .$$

Then the bijective $(A,\mrho_\Omega)$-algebra homomorphism $\eta: \algw(\Omega,C)\to \algt(\Omega,C)$ restricts to a bijection
$\eta: \ralgw(\Omega,C)\to \ralgt(\Omega,C)$. Thus we conclude

\begin{coro}
Let $(A,\mrho_\Omega)$ be an $\Omega$-operated algebra and let $B$ be a $\bfk$-augmented algebra. The pair $((\ralgt(\Omega,A\ot B),\tprod),\ropt_\Omega)$ is the free $(A,\mrho_\Omega)$-algebra over $B$.
\mlabel{co:freet}
\end{coro}

Applying Theorem~\mref{thm:freea} and Corollary~\mref{co:freet} to our subject of study, any integral equation can be regarded as an element of $\ralgt(\Omega,C)$ or $\ralgw(\Omega,C)$. Here $\Omega$ is the set of integral operators, including the integral limits and kernels. So for example $f(x)\mapsto \int_0^{x^2} K(x,t) f(t) \,dt$ is taken as one integral operator; $C=A\ot B$ is the algebra of coefficient functions $A$, such as $C(\RR)$, together with the variable functions $B$. In particular, let $Y$ be a set of variable functions. Then $C=C(\RR)[Y]$. By enlarging this $C$, we can also consider the case when transcendental functions of $Y$ are involved.

As the following examples illustrate, we have a three way dictionary among $\ralgt(\Omega,C)$, $\ralgw(\Omega,C)$ and integral expressions which, when set to zero, become integral equations.
\vspace{-.3cm}
{\tiny
\begin{equation}
\begin{array}{ccccccc}
\ctone{a}&
\cttwo{a}{b}{\alpha}&
\ctthree{a}{b}{c}{\alpha}{\beta} &
\ctfour{a}{b}{c}{d}{\alpha}{\beta}{\gamma} &
\ctfourb{a}{b}{c}{d}{\alpha}{\beta}{\gamma} &
\ctfourc{a}{b}{c}{d}{\alpha}{\beta}{\gamma} &
\ctfive{a}{b}{c}{d}{e}{\alpha}{\beta}{\gamma}{\delta}\\
&&&&&&\\
a& a\ropw_\alpha(b) & a\ropw_\alpha(b)\ropw_\beta(c)&
a\ropw_\alpha(b)\ropw_\beta(c)\ropw_\gamma(d)&
a\ropw_\alpha(b)\ropw_\beta\Big( c\ropw_\gamma(d)\Big)&
a\ropw_\alpha\Big(b\ropw_\beta(c)\ropw_\gamma(d)\Big)&
a\ropw_\alpha(b)\ropw_\beta \Big(c\ropw_\gamma(d)\ropw_\delta(e)\Big)\\
&&&&&&\\
a & a(\int_\alpha b)& a(\int_\alpha b)(\int_\beta c)&
a(\int_\alpha b)(\int_\beta c)(\int_\gamma d)&
a(\int_\alpha b)\int_\beta\Big( c(\int_\gamma d)\Big)&
a\int_\alpha\Big(b(\int_\beta c)(\int_\gamma d)\Big)&
a(\int_\alpha b)\int_\beta \Big(c(\int_\gamma d)(\int_\delta e)\Big)
\end{array}
\notag\mlabel{eq:dict}
\end{equation}
}
Here $ \alpha, \beta, \gamma, \delta$ are in $\Omega$ and $a, b, c, d, e$ are in $ C$ and further in $C^+_A$ when decorating leaf vertices.

Now we can define the main notions of our study, including a formal definition of integral equations.

\begin{defn}\mlabel{de:eqn}
Let $(A,\mrho_\Omega)$ be an $\Omega$-operated algebra and $B$ be a $\bfk$-augmented algebra.
\begin{enumerate}
	\item
The $(A,\mrho_\Omega)$-algebra $\ralgw(\Omega,A\ot B)$ (resp. $\ralgt(\Omega,A\ot B)$) is called the {\bf operated polynomial algebra in words} (resp. {\bf trees}) with operator set $\Omega$,  coefficient algebra $A$ and variable algebra $B$. Elements (setting to zero) in $\ralgw(\Omega,A\ot B)$ and $\ralgt(\Omega,A \ot B)$ are called {\bf operator equations} in words (resp. trees).
\mlabel{it:eqng}
\item
Let $B=\bfk[Y]$ and hence $A\ot B=A[Y]$ where $Y$ is a set. Then $\ralgw(\Omega,A[Y])$ and $\ralgt(\Omega,A[Y])$ are called the {\bf operated polynomial algebra} with operator set $\Omega$, coefficient algebra $A$ and independent variables $Y$. Their elements (setting to zero) are called {\bf operator equations} with the same qualifications.
\mlabel{it:eqngy}
\item
\mlabel{it:eqny}
Further let $A=C(I)$ for an open interval $I\subseteq \RR$, $\mrho_\Omega$ a set of integral operators on $A$ and $Y$ a set of variable functions. Then $\ralgw(\Omega,C(I)[Y])$ and $\ralgt(\Omega,C(I)[Y])$ are called the {\bf integral polynomial algebra} with integral operator set $\Omega$, coefficient function space $C(I)$ and variable functions $Y$. Their elements (setting to zero) are called {\bf integral equations} with the same qualifications.
\end{enumerate}
\end{defn}

\begin{remark} \mlabel{rk:opalg3}
As noted in Remark~\mref{rk:opalg2}, taking $B$ to be an algebra containing non-polynomial functions of the variable functions $Y$, we can include more general operator equations than when $B$ is taken $A[Y]$. For example, the Thomas-Fermi equation in Eq.~\meqref{eq:tf}, rewritten in Eq.~\meqref{eq:tf2}, is not in $\ralgw(\Omega,C(I)[y])$, but is in $\ralgw(\Omega,C(I)\ot B)$ for $B=\RR[y,z]/(z^2-y^3)$.
\end{remark}

As a prototype of operator equations and integral equations thus defined, taking the operated algebra $(A,\mrho_\Omega)$ in Item~\meqref{it:eqngy} to be a differential algebra, we arrive at a notion of differential equations.

\begin{defn} \mlabel{de:diffeqn}
Let $(A,d_0)$ be a differential algebra. In particular, let $A$ be the differential algebra of smooth functions on an open interval $I\subseteq \RR$ with the usual derivation. The operated algebra $\ralgw(\{d_0\},A[Y])$ (resp. $\ralgt(\{d_0\},A[Y])$) is called the {\bf differential polynomial algebra over $Y$}. Its elements (setting to zero) are called differential equations in $Y$.
\end{defn}

These notions of the differential polynomial algebra and a differential equation is apparently more general than the usual notions in Eq.~\meqref{eq:diffpoly} in the usual sense of differential algebra~\mcite{Ko,PS}. However, as to be shown in Proposition~\mref{pp:diffeq}, the two sets of notions are equivalent.

\subsection{Evaluations, solutions and equivalence of operator equations}

As defined in Definition~\mref{de:eqn}, for a given element $\phi\in \ralgw(\Omega,A \ot B)$, the equation
$$ \phi=0$$
is called an operator equation.
We are interested in the solutions of $\phi$ in the operated algebra $(A,\mrho_\Omega)$ (or its extension), when the operator $\ropw_\omega$ on $\ralgw(\Omega,A \ot B)$ is taken to be the specific linear operator $\mrho_\omega$ on $A, \omega\in \Omega$. By the universal property of $\ralgw(\Omega,A \ot B)$, an algebra homomorphism $f_B:B\to A$ gives rise to a unique $\Omega$-operated algebra homomorphism, called an {\bf evaluation map}:
$$\free{f}:=\free{f}_B:=\eva\big|_{f_B, \ropw_\Omega\mapsto \mrho_\Omega}:
\ralgw(\Omega,A \ot B) \to A,$$
such that $f=\free{f} i_B$.
Of particular interest to us is when $B=\bfk[Y]$ is a polynomial algebra, so $f_B:B\to A$ is uniquely determined by a map $f_Y:Y\to A$, denoted by $Y\mapsto f_Y$.
Then the equation $\phi=0$ has a solution $f_Y$ when $\ropw_\omega$ are taken to be the operators $\mrho_\Omega$ precisely means
$$\eva\big|_{Y\mapsto f_Y, \ropw_\Omega\mapsto \mrho_\Omega}(\phi)=0.$$

Thus the operator equations $\phi=0$ for which $f_Y$ is a solution form the kernel $J_{\mrho_\Omega,f_Y}$ of $\free{f}$, which is an operated ideal of $\ralgw(\Omega,A[Y])$. For a given set of linear operators $\mrho_\Omega$ on $A$, as we run through all the evaluations of $Y$, consider the intersection
\begin{equation}
J_{\mrho_\Omega}:=\bigcap_{Y\mapsto f_Y} J_{\mrho_\Omega,f_Y}.
\mlabel{eq:kernalpha}
\end{equation}
Elements in $J_{\mrho_\Omega}$ are precisely the (algebraic) operator identities satisfied by elements of $A$ and the operators $\mrho_\Omega$. Thus we have
\begin{prop}\mlabel{pp:reduce}
 \begin{enumerate}
 \item
Elements in $\ralgw(\Omega,A[Y])\backslash J_{\mrho_\Omega}$ give operator equations that do not hold identically in $(A,\mrho_\Omega)$.
\item
Two operator equations $\phi=0$ and $\psi=0$ for $\phi, \psi\in \ralgw(\Omega,A[Y])$ have the same solution set in $(A,\mrho_\Omega)$ if the operators differ by an element in $J_{\mrho_\Omega}$.
\end{enumerate}
\end{prop}
These properties lead us to define

\begin{defn}Let $(A,\mrho_\Omega)$ be an operated algebra and $Y$ be a set.
\begin{enumerate}
\item
Nonzero elements of $\ralgw(\Omega,A[Y])/J_{\mrho_\Omega}$ are called {\bf reduced operator polynomials} for $(A,\mrho_\Omega)$.
\item
Operated polynomials $\phi$ and $\psi$ (or their corresponding equations) in $\ralgw(\Omega,A[Y])$ are called {\bf equivalent} if $\phi-\psi$ is in $J_{\mrho_\Omega}$.
\end{enumerate}
\mlabel{de:intequiv}
\end{defn}
Even though it is usually not practical to display a canonical basis of the quotient $\ralgw(\Omega,\mrho_\Omega)/J_{\mrho_\Omega}$, it is possible to do so when the operated ideal $J_{\mrho_\Omega}$ is replaced by an operated subideal generated by defining identities of algebraically defined operators such as the differential operator (in differential algebra) or the Rota-Baxter operator, as we will show next.

We will consider the case when the operators $\mrho_\Omega$ are separable Volterra operators in the next section to obtain further properties of the integral polynomials and the corresponding integral equations, namely that any such integral equation is equivalent to an integral equation that is operator linear (Definition~\mref{de:oplin}).

As a prototype of this case, we first treat differential operators and differential equations.

\begin{prop} \mlabel{pp:diffeq}
Let $(A,\mrho_\Omega)$ be a differential algebra and let $Y$ be a set. Any differential polynomial in the sense of Definition~\mref{de:eqn}.\meqref{it:eqng}, that is, as an element of  $\ralgw(\{d\},A[Y])$, is equivalent to a differential polynomial in the usual sense of differential algebra and differential equation, that is, as an element of $A_\Omega\{Y\}$ in Eq.~\meqref{eq:diffpoly}.
\end{prop}

\begin{proof}
Let $\phi=\phi(\ropw_\Omega,A,Y)$ in the free $(A,\mrho_\Omega)$-operated algebra $(\ralgw(\Omega,A[Y]),\ropw_\Omega)$.
Consider the defining relations of differential operators in a differential algebra:
\begin{equation}
\ropw_\alpha(uv)-\ropw_\alpha(u)v-u\ropw_\alpha(v), \tforall u, v\in \ralgw(\Omega,A[Y]), \alpha \in \Omega,
\mlabel{eq:diffrel}
\end{equation}
in $\ralgw(\Omega,A[Y])$.

Since $\mrho_\Omega$ are differential operators on $A$,
for any $f: Y\to A$,
the induced $(A,\mrho_\Omega)$-algebra homomorphism $\free{f}$ sends the expressions in Eq.~\meqref{eq:diffrel} to zero. Thus the operated ideal $I_{\rm diff}$ generated by these expressions is contained in the ideal $J_{\mrho_\Omega}$ defined in Definition~\mref{de:intequiv}.

By definition, the quotient $(A,\mrho_\Omega)$-algebra $\ralgw(\Omega,A[Y])/I_{\rm diff}$ is the free $(A,\mrho_\Omega)$-differential algebra. Thus by the fact that $A_\Omega\{Y\}$ in Eq.~\meqref{eq:diffpoly} is also the free $(A,\mrho_\Omega)$-differential algebra, we have the natural isomorphism
$$\ralgw(\Omega,A[Y])/I_{\rm diff} \cong A_\Omega\{Y\}$$
of differential algebras. In fact, the isomorphism (or its inverse) is given linearly by sending a differential monomial $(y,\prod_{\omega\in \Omega}\omega^{n_\omega})\in Y\times C(\Omega)$ in $A_\Omega\{Y\}=A[Y\times C(\Omega)]$ to the corresponding element $\prod_{\omega\in \Omega}\Pi_\omega^{n_\omega}(y)$ in $\ralgw(\Omega,A[Y])$. This means that, in particular, the element $\phi(\ropw_\Omega,A,Y)$ in $\ralgw(\Omega,A[Y])$ is congruent modulo $I_{\rm diff}$ to an element of $A_\Omega\{Y\}$. Since $I_{\rm diff}$ is contained in  $J_{\mrho_\Omega}$, by Definition~\mref{de:intequiv}, $\phi(\ropw_\Omega,A,Y)$ is equivalent to a differential equation in $\ralgw(\Omega,A[Y])$ of the form given in $A_\Omega\{Y\}$.
\end{proof}

\section{Separable Volterra equations and their linearity}
\mlabel{sec:sep}
In this section, we first construct free relative \mtrbas which serve as the universal context for separable Volterra equations. We then apply the construction to prove the operator linearity of separable Volterra equations.

\subsection{Construction of free relative \mtrbas}

We start with the related notions.
As the analogy of an algebra over a ring in the context of \mtrbas, we give

\begin{defn}
	Fix an
	$\Omega$-\mtrba $(A,\mrho_\Omega,\mfraka_\Omega)$. By an {\bf $(A,\mrho_\Omega,\mfraka_\Omega)$-\mtrba}, we mean an $\Omega$-\mtrba $(R,P_\Omega,\rtwist_{R,\Omega})$ with an $\Omega$-\mtrba homomorphism $\stmap:=\stmap_R:(A,\mrho_\Omega,\mfraka_\Omega)\to (R,P_\Omega,\rtwist_{R,\Omega})$. We also use $(R,P_\Omega,\rtwist_{R,\Omega},\stmap_R)$ to denote the data.
	For $(A,\mrho_\Omega,\mfraka_\Omega)$-\mtrbas $(R_1,P_{1,\Omega},\rtwist_{1,\Omega},\stmap_1)$ and  $(R_2,P_{2,\Omega},\rtwist_{2,\Omega},\stmap_2)$, an $\Omega$-\mtrba homomorphism $f:(R_1,P_{1,\Omega},\rtwist_{1,\Omega})\to(R_2,P_{2,\Omega},\rtwist_{2,\Omega})$ is called an {\bf $(A,\mrho_\Omega,\mfraka_\Omega)$-\mtrba homomorphism} if $\stmap_2=f \stmap_1$.
\end{defn}
We will determine the free objects in the category of $(A,\mrho_\Omega,\mfraka_\Omega)$-\mtrbas, defined precisely as follows.
\begin{defn}
	Given an $\Omega$-\mtrba $(A,\mrho_\Omega,\mfraka_\Omega)$ and an augmented $\bk$-algebra $B$, a {\bf free $(A,\mrho_\Omega,\mfraka_\Omega)$-\mtrba} over $B$, denoted by the quintuple $(\frmtrba(A,B),P_{F,\Omega},\rtwist_{F,\Omega},\stmap_F,i_B)$, is an $(A,\mrho_\Omega,\mfraka_\Omega)$-\mtrba $(\frmtrba(A,B),P_{F,\Omega}, \rtwist_{F,\Omega},\stmap_F)$ with a $\bk$-algebra homomorphism $i_B:B\rar \frmtrba(A,B)$ satisfying the following universal property.
	For any $(A,\mrho_\Omega,\mfraka_\Omega)$-\mtrba $(R,P_{R,\Omega},\rtwist_{R,\Omega},\stmap_R)$ and a $\bk$-algebra homomorphism $f:B\rar R$, there exists a unique $(A,\mrho_\Omega,\mfraka_\Omega)$-\mtrba homomorphism $\freet{f}:(\frmtrba(A,B),P_{B,\Omega},\rtwist_{F,\Omega},\stmap_F)\to (R,P_{R,\Omega},\rtwist_{R,\Omega},\stmap_R)$ such that $f=\freet{f} i_B$.
\end{defn}
\vspace{-.2cm}

In the case when $B$ is the polynomial algebra $\bfk[Y]$ on a set $Y$ or the symmetric algebra $S(V)$ over a $\bfk$-module $V$, we obtain the notions of a free $(A,\mrho_\Omega,\mfraka_\Omega)$-\mtrba over $Y$ or $V$.

When the twist $\mfraka_\Omega$ is trivial, namely $\mfraka_\omega=1_A$ for $\omega\in\Omega$, we recover the notion of a free $(A,\mrho_\Omega)$-MRBA recalled in Definition~\mref{de:rbr}. See~\mcite{GGZy} where the construction is also provided. We will apply this construction to obtain the free $(A,\mrho_\Omega,\mfraka_\Omega)$-\mtrba over $B$.

We first give some notations. For a linear operator $Q$ on an algebra $R$ and invertible $a\in R$, denote $Q^a(r):=aQ(a^{-1}r)$. Then for a family of operators $Q_\Omega$ and invertible elements $a_\Omega$, denote
$$Q_\Omega^{a_\Omega}:=\Big\{Q_\omega^{a_\omega}\,|\, \omega\in\Omega\Big\}, \quad
Q_\Omega^{a_\Omega^{-1}}:=\Big\{Q_\omega^{a_\omega^{-1}}\,|\, \omega\in\Omega\Big\}.$$

\begin{prop} \mlabel{pp:freerel}
Let $(A,\mrho_\Omega,\mfraka_\Omega)$ be a fixed
$\Omega$-\mtrba. Let $(A,\mrho_\Omega^{\mfraka_\Omega^{-1}})$ be the $\Omega$-MRBA by Proposition~\mref{pp:twrbo}. Let $B=\bfk\oplus B^+$ be an augmented $\bk$-algebra.
Let $(\frmrba(A,B),P_{F,\Omega},\stmap_F,i_B)$ with $i_B:B\to \frmrba(A,B)$ be the free $(A,\mrho_\Omega^{\mfraka_\Omega^{-1}})$-MRBA over $B$.\footnote{Only the universal property is needed here. The construction will be recalled next.} Denote $\mfraka_{F,\Omega}:=j_F(\mfraka_\Omega)\subseteq \frmrba(A,B)$. Then the quintuple $(\frmrba(A,B),P_{F,\Omega}^{\mfraka_{F,\Omega}},\mfraka_{F,\Omega},\stmap_F,i_B)$ is the free $(A,\mrho_\Omega,\mfraka_\Omega)$-MTRBA over $B$.
\end{prop}

\begin{proof}
Let $(R,P_{R,\Omega},\mfraka_{R,\Omega},\stmap_R)$ be an $(A,\mrho_\Omega,\mfraka_\Omega)$-\mtrba and $f:B\to R$ an algebra homomorphism. 
To verify the desired universal property of $(\frmrba(A,B),P_{F,\Omega}^{\mfraka_{F,\Omega}},\mfraka_{F,\Omega},\stmap_F,i_B)$, we will construct a unique homomorphism $\freet{f}: \frmrba(A,B) \to R$ of $(A,\mrho_\Omega,\mfraka_\Omega)$-\mtrbas such that $\freet{f} i_B=f$.

By Proposition~\mref{pp:twrbo}, $(\frmrba(A,B),P_{F,\Omega}^{\mfraka_{F,\Omega}}, \mfraka_{F,\Omega},\stmap_F)$ is an $\Omega$-\mtrba and $(R,P_{R,\Omega}^{\mfraka_{R,\Omega}^{-1}},\stmap_R)$ is an $\Omega$-\mrba.
Also, $P_{F,\Omega}\stmap_F=\stmap_F\mrho_\Omega^{\mfraka_\Omega^{-1}}=\mfraka_{F,\Omega}^{-1}\stmap_F\mrho_\Omega\mfraka_\Omega$ means that
$\stmap_F\mrho_\Omega=\mfraka_{F,\Omega}P_{F,\Omega}\stmap_F\mfraka_\Omega^{-1}=P_{F,\Omega}^{\mfraka_{F,\Omega}}\stmap_F$, and $\stmap_R\mrho_\Omega=P_{R,\Omega}\stmap_R$ implies that
$\stmap_R\mrho_\Omega^{\mfraka_\Omega^{-1}}=P_{R,\Omega}^{\mfraka_{R,\Omega}^{-1}}\stmap_R$.
Hence,
$(\frmrba(A,B),P_{F,\Omega}^{\mfraka_{F,\Omega}}, \mfraka_{F,\Omega},\stmap_F)$ is an $(A,\mrho_\Omega,\mfraka_\Omega)$-\mtrba and $(R,P_{R,\Omega}^{\mfraka_{R,\Omega}^{-1}},\stmap_R)$ is an $(A,\mrho_\Omega^{\mfraka_\Omega^{-1}})$-\mrba.

By the universal property of the free $(A,\mrho_\Omega^{\mfraka_\Omega^{-1}})$-MRBA $(\frmrba(A,B),P_{F,\Omega},\stmap_F,i_B)$, the algebra homomorphism $f:B\to R$ induces a unique homomorphism of $(A,\mrho_\Omega^{\mfraka_\Omega^{-1}})$-MRBAs $\freem{f}: \frmrba(A,B)\to R$ such that $\freem{f} i_B=f$.
Thus we have the following diagram, where the left half is for the category of \mrbas and the right half is for the category of \mtrbas.
\vspace{-.3cm}

$$ \xymatrix{
& (\frmrba(A,B),P_{F,\Omega}) \ar[dd]^{\freem{f}} && (\frmrba(A,B),P_{F,\Omega}^{\mfraka_{F,\Omega}},\mfraka_{F,\Omega}) \ar[dd]^{\freet{f}}	& \\
(A,\mrho_\Omega^{\mfraka_\Omega^{-1}}) \ar[ur]^{j_F}\ar[dr]^{j_R} && B \ar[ul]^{i_B}\ar[dl]_{f} \ar[ur]^{i_B}\ar[dr]_f && (A,\mrho_\Omega,\mfraka_\Omega) \ar[ul]_{j_F}\ar[dl]_{j_R} \\
& (R,P_{R,\Omega}^{\mfraka_{R,\Omega}^{-1}}) && (R,P_{R,\Omega}, \mfraka_{R,\Omega}) &
}
$$

We show that $\freem{f}$ in the left half of the diagram also satisfies the universal property of $\freet{f}$ in the right half of the diagram. First note that, for $\omega\in \Omega$, we have
$$ \freem{f}(\mfraka_{F,\omega})=\freem{f} \stmap_F (\mfraka_\omega)=\stmap_R(\mfraka_\omega) =\mfraka_{R,\omega} \quad \tforall \omega\in \Omega.$$
Thus as functions,
$\mfraka_{R,\omega} \freem{f} \mfraka_{F,\omega}^{-1} =\freem{f}$ for $\omega\in \Omega.$
We then obtain
\begin{eqnarray*}
\freem{f} P_{F,\omega}^{\mfraka_{F,\omega}} -P_{R,\omega}\freem{f}
&=& \mfraka_{R,\omega}\freem{f} \mfraka_{F,\omega}^{-1} \mfraka_{F,\omega} P_{F,\omega}\mfraka_{F,\omega}^{-1}-P_{R,\omega}\mfraka_{R,\omega} \freem{f}\mfraka_{F,\omega}^{-1}\\
&=& \mfraka_{R,\omega} \freem{f} P_{F,\omega} \mfraka^{-1}_{F,\omega}-P_{R,\omega} \mfraka_{R,\omega} \freem{f} \mfraka_{F,\omega}^{-1} \\
&=& \mfraka_{R,\omega}(\freem{f} P_{F,\omega} - \mfraka_{R,\omega}^{-1}P_{R,\omega} \mfraka_{R,\omega} \freem{f})\mfraka_{F,\omega}^{-1}\\
&=& \mfraka_{R,\omega}(\freem{f} P_{F,\omega}- P_{R,\omega}^{\mfraka_{R,\omega}^{-1}} \freem{f}) \mfraka_{F,\omega}^{-1},
\end{eqnarray*}
which vanishes since $\freem{f}$ is a homomorphism of \mrbas. Thus $\freem{f}$ is compatible with the linear operators and hence is a homomorphism of $(A,\mrho_\Omega,\mfraka_\Omega)$-\mtrbas.

Suppose $\freet{f}': (\frmrba(A,B),P_{F,\Omega},\rtwist_{F,\Omega},\stmap_F)\to (R,P_{R,\Omega},\rtwist_{R,\Omega},\stmap_R)$ is another
$(A,\mrho_\Omega,\mfraka_\Omega)$-\mtrba homomorphism such that $f=\freet{f}' i_B$. Then by the same argument, we find that $\freet{f}'$ is an $(A,\mrho_\Omega^{\mfraka_\Omega^{-1}})$-\mrba homomorphism such that $f=\freet{f}' i_B$. Then we must have $\freet{f}'=\freem{f}$.
This completes the proof.
\end{proof}

Now we can apply Proposition~\mref{pp:freerel} to construct the free $(A,\mrho_\Omega,\mfraka_\Omega)$-\mtrba over $B$ via the construction of the free $(A,\mrho_\Omega^{\mfraka_\Omega^{-1}})$-\mrba over $B$ given in~\mcite{GGZy}. We recall this construction without the abbreviations used there. 

For distinction, let $(A,\phi_\Omega)$ denote a fixed $\Omega$-\mrba. Let $B=\bfk\oplus B^+$ be an augmented $\bk$-algebra. Denote
$$\frakA:=A\ot B, \quad \frakA^+:=A\ot B^+.$$
Then $\frakA=A\oplus\frakA^+$. 

Consider the $\bfk$-module
\begin{equation} \mlabel{eq:fmrba}
	\trsha(A,B):=\bigoplus_{k\geq0}\frakA\ot (\bfk\Omega\ot \frakA^+)^{\ot k}.
\end{equation}

We first define a system of linear operators $P_{F,\omega},\,\omega\in\Omega$ on $\trsha(A,B)$. Any pure tensor of $\trsha(A,B)$ is of the form
 $$\fraku=u_0\ot (\omega_1\ot u_1)\ot(\omega_2\ot u_2)\cdots\ot (\omega_n\ot u_n), u_0\in A, u_i\in \frakA^+, \omega_i\in \Omega,\, 1\leq i\leq n, n\geq0.$$ 
We then define

\begin{equation}
	P_{F,\omega}(\fraku):=\left\{\begin{array}{ll}
		\phi_\omega(u_0),&u_0\in A,n=0,\\
		\phi_\omega(u_0)\ot (\omega_1\ot u_1)\ot (\omega_2\ot u_2)\cdots\ot (\omega_n\ot u_n)& \\\quad -1_A\ot (\omega_1\ot \phi_{\omega}(u_0)u_1)\ot (\omega_2\ot u_2)\ot \cdots \ot (\omega_n\ot u_n), & u_0\in A,n>0,\\
		1_A\ot (\omega\ot u_0)\ot (\omega_1\ot u_1)\ot\cdots\ot (\omega_n\ot u_n), & u_0\in \frakA^+.
	\end{array}\right.
	\mlabel{eq:rwro0}
\end{equation}

By means of these operators, we can express a pure tensor $\fraku=u_0\ot (\omega_1\ot u_1)\ot\cdots\ot (\omega_n\ot u_n), n>0,$ of $\trsha(A,B)$ in the recursive form
$$\fraku=u_0\Big(1_A\ot (\omega_1\ot u_1)\ot \cdots \ot (\omega_n\ot u_n)\Big)
=u_0 P_{F,\omega_1} (\fraku') 
 \text{ with } \fraku':= u_1\ot(\omega_2\ot u_2)\cdots\ot (\omega_n\ot u_n).$$
We can then recursively define a binary operation $\reypr=\reypr_\Omega$ on $\trsha(A,B)$. For any pure tensors 
$$\fraku=u_0\ot (\omega_1\ot u_1)\ot \cdots (\omega_m\ot u_m)=u_0P_{F,\omega_1}(\fraku')\in \frakA\ot (\bfk\Omega\ot \frakA^+)^{\ot m}$$ 
and 
$$\frakv=v_0\ot (\vep_1\ot v_1)\ot \cdots (\vep_n\ot v_n) = v_0P_{F,\vep_1} (\frakv')\in \frakA\ot (\bfk\Omega\ot \frakA^+)^{\ot n}$$ 
with their recursive forms when $m,n>0$, define 
{\small
	\begin{equation}
			\fraku\reypr\frakv:=\fraku\reypr_\Omega\frakv:=\begin{cases}
				u_0v_0, & \text{if } m=n=0, \\
				u_0v_0\ot (\vep_1\ot v_1)\ot \cdots \ot (\vep_n\ot v_n)=u_0v_0P_{F,\vep_1}(\frakv'),&\text{if }m=0,\,n>0,\\
				u_0v_0 \ot (\omega_1\ot u_1)\ot \cdots \ot (\omega_m\ot u_m)=u_0v_0P_{F,\omega_1}(\fraku'),&\text{if }m>0,\,n=0,\\	 
				u_0v_0P_{F,\vep_1}\Big(P_{F,\omega_1}(\fraku')\reypr \frakv'\Big)+
				u_0v_0P_{F,\omega_1}\Big(\fraku'\reypr P_{F,\vep_1}(\frakv')\Big), 
				&\text{if }m,n>0.
			\end{cases}
		\mlabel{eq:rwshp0}
	\end{equation}
}
Then extend $\reypr$ to all of $\trsha(A,B)$ by bilinearity. Let $j_F:A\to \trsha(A,B)$ be the natural inclusion.

\begin{theorem} $($\mcite{GGZy}$)$
	Let $(A,\phi_\Omega)$ be an $\Omega$-\mrba and $B$ an augmented $\bfk$-algebra. Then the quadruple $\Big((\trsha(A,B),\reypr),P_{F,\Omega},j_F,i_B\Big)$
	is the free
	$(A,\phi_\Omega)$-\mrba over $B$.
	\mlabel{th:rfcwr0}
\end{theorem}

Now let $(A,\mrho_\Omega,\mfraka_\Omega)$ be an $\Omega$-\mtrba. Applying Proposition~\mref{pp:freerel}, we can adopt Theorem~\mref{th:rfcwr0} to construct the free $(A,\mrho_\Omega,\mfraka_\Omega)$-\mtrba over $B$ as follows.
Take $\phi_\Omega:=\mrho_\Omega^{\mfraka_\Omega^{-1}}$ in Theorem~\mref{th:rfcwr0} to give the free $(A,\mrho_\Omega^{\mfraka_\Omega^{-1}})$-MRBA over $B$ as $\Big((\trsha(A,B),\reypr),P_{F,\Omega},j_F,i_B\Big)$. Then with the operator $P_{F,\Omega}$ defined in  Eq.~\meqref{eq:rwro0}, the operator $P_{F,\Omega}^{\mfraka_{F,\Omega}}$ on $\trsha(A,B)$ given in Proposition~\mref{pp:freerel} becomes
\begin{equation}
	 P_{F,\Omega}^{\mfraka_{F,\Omega}}(\fraku):=\left\{\begin{array}{ll}
		\mrho_\omega(u_0),&u_0\in A,n=0,\\
		 \mrho_\omega(u_0)\ot (\omega_1\ot u_1)\ot \cdots\ot (\omega_n\ot u_n)& \\\quad -\mfraka_\omega\ot (\omega_1\ot \mfraka_\omega^{-1}\mrho_\omega(u_0)u_1)\ot (\omega_2\ot u_2)\ot\cdots \ot (\omega_n\ot u_n), & u_0\in A,n>0,\\
		\mfraka_\omega\ot (\omega\ot \mfraka_\omega^{-1}u_0)\ot(\omega_1\ot u_1)\ot\cdots\ot (\omega_n\ot u_n), & u_0\in \frakA^+.
	\end{array}\right.
	\mlabel{eq:rwro}
\end{equation}
for any $\fraku=u_0\ot (\omega_1\ot u_1)\ot \cdots\ot (\omega_n\ot u_n)\in\frakA\ot(\bfk\Omega\ot \frakA^+)^{\ot n},\,n\geq0$.

Then by Proposition~\mref{pp:freerel}, we obtain
\begin{theorem}
	Let $(A,\mrho_\Omega,\mfraka_\Omega)$ be an $\Omega$-\mtrba and $B$ an augmented $\bfk$-algebra. Then with the $\bfk$-module $\trsha(A,B)$ in Eq.~\meqref{eq:fmrba}, the multiplication $\reypr$ in Eq.~\meqref{eq:rwshp0}, the linear operators $P_{F,\Omega}$ in Eq.~\meqref{eq:rwro} and $\mfraka_{F,\Omega}:=j_F(\mfraka_\Omega)\subseteq \frmrba(A,B)$, the quintuple $\Big((\trsha(A,B),\reypr_\Omega),P_{F,\Omega}^{\mfraka_{F,\Omega}}, \mfraka_{F,\Omega},\stmap_F,i_B\Big)$
	is the free
	$(A,\mrho_\Omega,\mfraka_\Omega)$-\mtrba over $B$.
	\mlabel{th:rfcwr}
\end{theorem}

\subsection{Operator linearity of separable Volterra equations}

We now apply the construction of free \mtrbas to obtain the operator linearity (Definition~\mref{de:oplin}) of integral equations with separable Volterra operators.

\begin{theorem}	\mlabel{thm:oplin}
	
Let $I\subseteq \RR$ be an open interval and $A:=C(I)$. Let $K_\Omega:=(K_\omega(x,t)\,|\,\omega\in \Omega)$ be a family of separable kernels $K_\omega(x,t)=k_\omega(x)h_\omega(t)\in C(I^2)$ with $k_\omega(x)$ zero free on $I$.
\delete{
Let $P_{K_\Omega}$ be the corresponding family of Volterra operators $P_{K_\omega}(f):=\int_a^xK_\omega(x,t)f(t)dt$ with $a\in I$ and let 
$\mfraka_\omega:=\tfrac{k_\omega(x)}{k_\omega(a)}, \omega\in\Omega$.  
}
Let $(A,P_{K_\Omega},\mfraka_\Omega)$ be the \mtrba in Theorem~\mref{thm:twfunction}. For any given set $Y$, any element (and its corresponding integral equation) in the free $(A,\mrho_\Omega)$-algebra $(\ralgw(\Omega,A[Y]),\ropw_\Omega)$ over $B$ defined in Theorem~\mref{thm:freea} is equivalent to one that is \oplin and in which each of the iterated operators acts on variable functions $($that is, on functions in $A[Y]^+$$)$.
\end{theorem}

As noted after Definition~\mref{de:intequiv}, this theorem shows that in order to study integral equations of a Volterra operator with separable kernels, we only need to consider operator linear integral equations with these Volterra operators.

\begin{proof}
	Consider the defining relation of the $\Omega$-matching \tw Rota-Baxter operator
\begin{equation}\mlabel{eq:rrb2}
			 \ropw_\alpha(u)\ropw_\beta(v)-\mfraka_\alpha\ropw_\beta\left(\mfraka_\alpha^{-1}\ropw_\alpha(u)v\right) -\mfraka_\beta\ropw_\alpha\left(\mfraka_\beta^{-1}u\ropw_\beta(v)\right) \tforall u, v\in \ralgw(\Omega,A[Y]), \alpha, \beta\in \Omega,
\end{equation}
	in $\ralgw(\Omega,A[Y])$.
	Since $(A,P_{K_\Omega},\mfraka_\Omega)$ is an $\Omega$-\mtrba, for any $f: Y\to A$, the induced $(A,P_{K_\Omega},\mfraka_\Omega)$-\mtrba homomorphism $\free{f}$ sends the expressions in Eq.~\meqref{eq:rrb2} to zero. Thus the operated ideal $I_{\rm TRB}$ generated by these expressions is contained in the ideal $J_{\mrho_\Omega}$ defined in Definition~\mref{de:intequiv}, where $\mrho_\Omega=P_{K_\Omega}$.
	
	By definition, the quotient $(A,P_{K_\Omega})$-algebra $\ralgw(\Omega,A[Y])/I_{\rm TRB}$ is the free $(A,P_{K_\Omega},\mfraka_\Omega)$-\mtrba. With the assumption on $(A,P_{K_\Omega},\mfraka_\Omega)$ in the theorem, by Theorem~\mref{th:rfcwr}, this free object is also given by $\trsha(A,B):=\Big((\trsha(A,B),\reypr), P_{F,\Omega}^{\mfraka_{F,\Omega}},\mfraka_{F,\Omega},j_F,i_B\Big)$ with $B=\RR[Y]$. Thus we have a natural isomorphism
	$$\ralgw(\Omega,A[Y])/I_{\rm TRB} \cong \trsha(A,\RR[Y])$$
	of $(A,P_{K_\Omega},\mfraka_\Omega)$-\mtrbas. This means that, in particular, any element $\phi(\ropw_\Omega,A,Y)$ in $\ralgw(\Omega,A[Y])$ is congruent modulo $I_{\rm TRB}$ to an element of $\trsha(A,\RR[Y])$. 
	Since 
	$$\trsha(A,\RR[Y]):=\bigoplus_{k\geq0}A[Y]\ot (\bfk\Omega\ot A[Y]^+)^{\ot k},$$ 
	the latter element is operator linear and each of the iterated operators acts on variable functions. Since $I_{\rm TRB}\subseteq J_{\mrho_\Omega}$, by Definition~\mref{de:intequiv}, $\phi(\ropw_\Omega,A,Y)$ is equivalent to an integral equation in $\ralgw(\Omega,A[Y])$ that is operator linear with the mentioned additional property.
\end{proof}

We next apply the theorem to some examples of Volterra equations, beginning with the ones in Example~\mref{ex:classic} that motivated our study.

\begin{exam}
	\begin{enumerate}
		\item
		We first note that the Volterra population model in Example~\mref{ex:classic}~ \meqref{it:volt} is already \oplin.
		\item
		The Thomas-Fermi equation from Example~\mref{ex:classic}, Eq.~\meqref{eq:tf} is already \oplin. But there is no variable function between the outer and inner integrals, as prescribed in Theorem~\mref{thm:oplin} (by taking the independent variable to be $z=y^{1/2}$). Integration by parts gives
		$$\int_0^x\int_0^t s^{-1/2} y(s)^{3/2} ds\,dt
		= x\int_0^x s^{-1/2}y(s)^{3/2} ds -\int_0^x t^{1/2}y(t)^{3/2}dt,$$
		now \oplin with the desired form.
	\end{enumerate}
\end{exam}

As shown in Theorem~\mref{thm:twfunction}, the Volterra integral operator
\[P_K(f) := k(x)\int_a^xh(t)f(t)dt\]
is a \tw Rota-Baxter operator by $\mfraka=\frac{k(x)}{k(a)}$, and thus satisfies the twisted Rota-Baxter identity
\[P_K(f)P_K(g)=\mfraka P_K\left(\mfraka^{-1}(P_K(f)g+fP_K(g))\right)\]
for any $f(x),g(x) \in C(I)$.
We apply this identity in the following Volterra equations to obtain operator linear forms.
Naturally, these examples are guaranteed to work thanks to Theorem~\mref{thm:oplin}.
\begin{exam}
\begin{enumerate}
		\item Let the kernel be $K(x,t) = e^{-x+t}$ and $a=0$, so $\mfraka = e^{-x}$.  On $C(\RR)$, the Volterra equation
		{\small $$0 = \left(\int_0^xe^{-x+t}f(t)\int_0^te^{-t+u}g(u)\,du\,dt\right)\left(\int_0^xe^{-x+t}h(t)\,dt\right) - \left(\int_0^xe^{-x+t}f(t)\int_0^te^{-t+u}h(u)\,du\,dt\right)\left(\int_0^xe^{-x+t}g(t)\,dt\right)$$
		}
		written in operator form is
{\small		\begin{align*}
			0 &= P_K(fP_K(g))P_K(h) - P_K(fP_K(h))P_K(g) \\
			&= \mfraka P_K(\mfraka^{-1}(P_K(fP_K(g))h + fP_K(g)P_K(h))) - \mfraka P_K(\mfraka^{-1}(P_K(fP_K(h))g + fP_K(h)P_K(g))) \\
			&= \mfraka P_K(\mfraka^{-1}hP_K(fP_K(g)))-\mfraka P_K(\mfraka^{-1}gP_K(fP_K(h))).
		\end{align*}
	}
Rewriting this back using the integral notations, we obtain
		{\small \begin{align*}
				0 &=
e^{-x}\int_0^xe^{-x+t}\left(e^t h(t)\int_0^te^{-t+u}\left(f(u)\int_0^ue^{-u+s}g(s)\,ds\right)\,du\right)\,dt\\
&\hskip1cm
- e^{-x}\int_0^xe^{-x+t}\left(e^t g(t)\int_0^te^{-t+u}\left(f(u)\int_0^ue^{-u+s}h(s)\,ds\right)\,du\right)\,dt.
			\end{align*}
}
		Observe that this equation is operator linear and each of the iterated integrals acts on the variable functions $f, g$ and $h$.
\item For $K(x,t) = xt$ and $a=1$, we have $\mfraka = x$.  On $C((1,\infty))$, we can rewrite the equation
		\begin{align*}
			f(x) &= \left(\int_1^x xtf(t)g(t)\,dt\right)\left(\int_1^x xtf(t)\int_1^t tuh(u)\,du\,dt\right)
\end{align*}
as
	\begin{align*}
f	&= P_K(fg)P_K(fP_K(h)) \\
			&= \mfraka P_K\left(\mfraka^{-1}(P_K(fg)fP_K(h) + fgP_K(fP_K(h)))\right) \\
			&=\mfraka P_K\left(fP_K\left(\mfraka^{-1}hP_K(fg)\right)\right) + \mfraka P_K\left(fP_K\left(\mfraka^{-1}fgP_K(h)\right)\right) + \mfraka P_K\left(\mfraka^{-1} fgP_K(fP_K(h))\right)
\end{align*}
Again, the last expression rewrites to integrals in operator linear forms and each of the iterated integrals acts on the variable functions $f, g$ and $h$.
	\end{enumerate}
\end{exam}

\noindent
{\bf Acknowledgments.} This work is supported by Natural Science Foundation of China (Grant Nos. 12071094, 11771142,  11771190) and the China Scholarship Council (No. 201808440068).
Y. Li thanks Rutgers University -- Newark for its hospitality during his visit in 2018-2019.

\end{document}